\documentclass[a4paper,12pt]{article}

\usepackage{subfiles}

\usepackage{amsfonts,amssymb,amsmath,amsthm,latexsym,epsfig,amscd,bbm}

\usepackage[english]{babel}
\usepackage{graphicx}
\usepackage{tikz}
\usepackage{courier}
\usepackage{bbm}
\usepackage{enumerate}
\usepackage{psfrag}
\usepackage{wasysym}
\usepackage{type1cm}

\usepackage[notref,notcite,color]{showkeys}

\usepackage{calc}

\def\arraypar#1{\parbox[c]{\textwidth - 2cm}{\centering #1}}

\usepackage{environ}
\NewEnviron{display}{
\begin{equation}\begin{array}{c}
\arraypar{\BODY}
\end{array}\end{equation}
}

\def\clap#1{\hbox to 0pt{\hss#1\hss}}

\makeatletter \@addtoreset{equation}{section}
\makeatother

\makeatletter \@addtoreset{enunciato}{section}
\makeatother

\newcounter{enunciato}[section]

\newtheorem{ittheorem}{Theorem}
\newtheorem{itlemma}{Lemma}
\newtheorem{itproposition}{Proposition}
\newtheorem{itdefinition}{Definition}
\newtheorem{itremark}{Remark}
\newtheorem{itclaim}{Claim}
\newtheorem{itfact}{Fact}
\newtheorem{itconjecture}{Conjecture}
\newtheorem{itcorollary}{Corollary}

\newenvironment{theorem}{\addtocounter{enunciato}{1}
\begin{ittheorem}}{\end{ittheorem}}

\newenvironment{lemma}{\addtocounter{enunciato}{1}
\begin{itlemma}}{\end{itlemma}}

\newenvironment{proposition}{\addtocounter{enunciato}{1}
\begin{itproposition}}{\end{itproposition}}

\newenvironment{definition}{\addtocounter{enunciato}{1}
\begin{itdefinition}}{\end{itdefinition}}

\newenvironment{remark}{\addtocounter{enunciato}{1}
\begin{itremark}}{\end{itremark}}

\newenvironment{conjecture}{\addtocounter{enunciato}{1}
\begin{itconjecture}}{\end{itconjecture}}

\newenvironment{corollary}{\addtocounter{enunciato}{1}
\begin{itcorollary}}{\end{itcorollary}}

\newcommand{\be}[1]{\begin{equation}\label{#1}}
\newcommand{\ee}{\end{equation}}

\newcommand{\bl}[1]{\begin{lemma}\label{#1}}
\newcommand{\el}{\end{lemma}}

\newcommand{\br}[1]{\begin{remark}\label{#1}}
\newcommand{\er}{\end{remark}}

\newcommand{\bt}[1]{\begin{theorem}\label{#1}}
\newcommand{\et}{\end{theorem}}

\newcommand{\bd}[1]{\begin{definition}\label{#1}}
\newcommand{\ed}{\end{definition}}

\newcommand{\bp}[1]{\begin{proposition}\label{#1}}
\newcommand{\ep}{\end{proposition}}

\newcommand{\bc}[1]{\begin{corollary}\label{#1}}
\newcommand{\ec}{\end{corollary}}

\newcommand{\bcj}[1]{\begin{conjecture}\label{#1}}
\newcommand{\ecj}{\end{conjecture}}

\newcommand{\bpr}{\begin{proof}}
\newcommand{\epr}{\end{proof}}

\def\Z{\mathbb{Z}}
\def\N{\mathbb{N}}
\def\R{\mathbb{R}}

\def\P{\mathbb{P}}
\def\E{\mathbb{E}}

\def \ba {\begin{array}}
\def \ea {\end{array}}

\def \P  {{\mathbb P}}
\def \E  {{\mathbb E}}

\def \cE {{\mathcal E}}

\def \cW {{\mathcal W}}

\usepackage{graphicx}

\DeclareSymbolFont{symbolsC}{U}{pxsyc}{m}{n}
\DeclareMathSymbol{\opentimes}{\mathrel}{symbolsC}{93}

\newcommand{\um}{{\angle}}
\newcommand{\tres}{{\mathbin{\, \text{\rotatebox[origin=c]{180}{$\angle$}}}}}
\newcommand{\treze}{{\mathbin{\text{\rotatebox[origin=c]{35}{$\opentimes$}}}}}

\newcounter{constant}
\newcommand{\newconstant}[1]{\refstepcounter{constant}\label{#1}}
\newcommand{\useconstant}[1]{c_{\textnormal{\tiny \ref{#1}}}}
\usepackage{array}
\newcolumntype{e}{>{\displaystyle}r @{\,} >{\displaystyle}c @{\,} >{\displaystyle}l}

\usepackage[colorlinks]{hyperref}
\usepackage[figure,table]{hypcap} 
\hypersetup{final}
\usepackage{xcolor}
\hypersetup{
    colorlinks,
    linkcolor={blue!50!black},
    citecolor={blue!50!black},
    urlcolor={blue!80!black}
}

\usepackage[utf8]{inputenc}

\newcommand{\footremember}[2]{%
    \footnote{#2}
    \newcounter{#1}
    \setcounter{#1}{\value{footnote}}%
}
\newcommand{\footrecall}[1]{%
    \footnotemark[\value{#1}]%
} 

\title{Random walk on random walks: low densities}
\author{
  Oriane Blondel \footnote{CNRS, Univ Lyon, Université Claude Bernard Lyon 1, CNRS UMR 5208, Institut Camille Jordan, 43 boulevard du 11 novembre 1918 -- 69622, France}
  \and
  Marcelo R.\ Hilário \footnote{Universidade Federal de Minas Gerais, Dep. de Matemática, 31270-901 Belo Horizonte} \footremember{NYUSh}{NYU-Shanghai, 1555 Century Av., Pudong Shanghai, CN 200122}
    \and
  Renato S.\ dos Santos \footnote{Weierstrass Institute for Applied Analysis and Stochastics, Mohrenstr. 39, 10117 Berlin}
  \and
  Vladas Sidoravicius \footnote{Courant Institute, NYU, 251 Mercer Street New York, NY 10012} \footrecall{NYUSh}
  \and
  Augusto Teixeira \thanks{Instituto de Matemática Pura e Aplicada, Estrada Dona Castorina 110, 22460-320 Rio de Janeiro}
  }
\date{\today}

\begin{document}

\maketitle

\begin{abstract}
We consider a random walker in a dynamic random environment given by a system of independent
simple symmetric random walks.
We obtain ballisticity results under two types of perturbations: 
low particle density, and strong local drift on particles.
Surprisingly, the random walker may behave very differently 
depending on whether the underlying environment particles perform lazy or non-lazy random walks,
which is related to a notion of permeability of the system.
We also provide a strong law of large numbers, a functional central limit theorem
and large deviation bounds under an ellipticity condition.
\end{abstract}


\section{Introduction and main results}
\label{s:intro}
The present article is a continuation of the works \cite{manyRW, HHSST14}
concerning the behaviour of a random walker in a dynamic random environment (RWDRE)
given by a system of independent simple symmetric random walks.
These works are focused on the high density regime in one and higher dimensions,
respectively.
Here we will consider the low density regime  in one dimension,
and also the case of a strong local drift on particles.
As indicated in \cite{manyRW, HHSST14},
the main challenge in this model stems from the relatively poor mixing properties of the random environment.
In fact, these properties become even worse as the density decreases,
which poses additional difficulties in our setting.
A brief overview of connections to the literature will be given in
Section~\ref{ss:connections} below.

Let us introduce the environment over which we will define our random walker.
Let $\Z_+ := \N \cup \{0\}$ where $\N$ is the set of positive integers.
Fix $\rho >0$ and let $(N(x,0))_{x \in \Z}$ be an i.i.d.\ collection of Poisson($\rho$) random variables.
Let $(S^{z,i})_{z \in \Z, i \in \N}$ be a collection of simple symmetric random walks on $\Z$,
independent of $(N(x,0))_{x \in \Z}$ and such that $(S^{z,i}-z)_{z \in \Z, i \in \N}$ are centered,
independent and identically distributed.
We call $S^{z,i}$ with $i\le N(z,0)$ a \emph{particle}.
We then let $N(x,t) := \sum_{z \in \Z, i\le N(z,0)} \mathbbm{1}_{\{S^{z,i}_t = x\}}$,
i.e., $N(x,t)$ is the number of particles present at the space-time point $(x,t)$.

To define the random walker $X = (X_t)_{t \in \Z_+}$,
let $p_\circ, p_\bullet \in [0,1]$.
For a fixed realization of $N = (N(x,t))_{x \in \Z, t \in \Z_+}$,
 $X$ is defined as the time-inhomogeneous Markov chain on $\Z$ that starts at $0$ and,
when at position $x$ at time $t$, jumps to $x+1$ with probability
\begin{equation}\label{e:defX}
p_\circ \; \text{ if } \; N(x,t) = 0, \quad \text{ or } \quad
p_\bullet \; \text{ if } \; N(x,t) \ge 1,
\end{equation}
and jumps to $x-1$ otherwise.
The parameters $p_\circ, p_\bullet \in [0,1]$ thus
represent the chance for random walker to jump to the right in the absence (respectively, presence) of particles.
It will be also convenient to define the local drifts
\begin{equation}
v_\circ:=2p_\circ-1, \quad v_\bullet:=2p_\bullet -1.
\end{equation}
The case $v_\circ v_\bullet>0$ is called \emph{non-nestling} and has already been treated in \cite{HHSST14}.
Here, we will focus on the case
\begin{equation}\label{e:assumpvel}
v_\bullet \leq 0 < v_\circ,
\end{equation}
meaning that random walker has a local drift to the right on empty sites,
and no drift to the right on sites occupied by particles.

An important parameter in our analysis will be
\begin{equation}\label{e:defq_0}
q_0 := P(S^{0,1}_1 = 0) \in [0,1).
\end{equation}
When $q_0 > 0$ we say that the random walks $S^{z,i}$ are \emph{lazy}.

Surprisingly, the asymptotic behaviour of the random walker
may strongly depend on whether $q_0=0$ or $q_0 > 0$.
Indeed, for small values of $p_\bullet$,
the random walker may develop a positive speed if $q_0>0$ and a negative one if $q_0=0$.
This is related to a notion of \emph{permeability}:
if $p_\bullet=q_0=0$, the random walker cannot cross any particles that it meets to the right,
and we say that the system is \emph{impermeable} to the random walker.
If either $p_\bullet$ or $q_0$ are positive, it is possible for the walker to cross particles in both directions,
and we call the system \emph{permeable}.

Let $\P^\rho$ denote the joint law of $N$ and $X$ for a fixed density $\rho>0$.
In order to describe our results, we introduce the following condition:
\begin{definition}[Ballisticity condition]
\label{e:defBAL}
Fixed $\rho$, $p_\circ$, $p_\bullet$, $q_0$ and given $v_\star \neq 0$,
we say that the \emph{ballisticity condition with speed $v_\star$}
is satisfied if there exist $\gamma > 1$ and $c_1, c_2 \in (0,\infty)$ such that
\begin{equation}\label{e:BAL}
\P^{\rho} \Big(\exists \, n \in \N \colon \frac{v_\star}{|v_\star|} X_n < |v_\star| n - L \Big) \le c_1 \exp\left\{-c_2 (\log L)^\gamma \right\} \;\;\; \forall \; L \in \N.
\end{equation}
\end{definition}
\noindent
Condition~\eqref{e:BAL} is reminiscent of ballisticity conditions from the literature of random walks in static random environments
such as Sznitman's $(T')$ condition (cf.\ \cite{Sz02}).
Such a condition provides control on the backtracking probability of the random walker
that can be very useful in obtaining finer asymptotic results, see e.g.\ Theorem~\ref{t:limits_permeable} below.

Note that, if $\rho = 0$ (i.e., if no particles are present),
the random walker has a global drift $v_\circ$, which is positive under \eqref{e:assumpvel}.
Our first ballisticity result states that, in the permeable case,
perturbations around $\rho = 0$ still lead to a positive speed.
\begin{theorem}\label{t:ballisticity_permeable}
Assume \eqref{e:assumpvel} and $p_\bullet \vee q_0 > 0$.
There exist $\rho_\star = \rho_\star(p_\circ, p_\bullet, q_0) > 0$
and $v_\star = v_\star(p_\circ, p_\bullet, q_0) > 0$
such that, for any $\rho \le \rho_\star$,
\eqref{e:BAL} holds with $\gamma = 3/2$.
\end{theorem}

Our second ballisticity result shows a radically distinct behaviour
for perturbations of $p_\bullet$ around the impermeable case.
\begin{theorem}\label{t:ballisticity_impermeable}
Assume $q_0=0$.
For any $p_\circ \in [0,1]$, $\rho >0$ and $\gamma \in (1,3/2)$,
there exist $v_\star = v_\star(\rho) < 0$
and $p_\star =p_\star(p_\circ,\rho, \gamma) \in (0,1)$
such that, if $p_\bullet \le p_\star$, then \eqref{e:BAL} holds.
\end{theorem}

Theorem~\ref{t:ballisticity_impermeable} may be seen as a manifestation of particle conservation in our dynamic random environment.
Indeed, when $q_0=0$, this conservation forces the random walker to interact with environment particles that it crosses;
see Section~\ref{ss:trigger_imperm}.

The difference in the ballistic behaviour of the two cases is illustrated
by the phase diagrams in Figure~\ref{f:phase_diagram}.

\begin{figure}[h]
  \centering
  \begin{tikzpicture}[scale=.7]
    \draw[fill, color=gray!30!white] (2,2) .. controls (1.5,1.5) and (1,1) .. (0,.8) -- (0,0) -- (2,0) -- (2,2);
    \draw[fill, color=gray!30!white] (12,1.7) .. controls (11,1.4) and (10.5,1.2) .. (10,0) -- (10,2.2) -- (12,2.2) -- (12,1.7);
    \draw[<->] (0,4) -- (0,0) -- (5, 0);
    \draw[<->] (10,4) -- (10,0) -- (15, 0);
    \draw (4,-.2) -- (4,.2) (14,-.2) -- (14,.2);
    \node[right] at (0,5) {Lazy environment};
    \node[right] at (9.5,5) {Non-lazy environment};
    \node[below] at (0,-.2) {$0$};
    \node[below] at (10,-.2) {$0$};
    \node[below] at (4,-.2) {$1$};
    \node[below] at (14,-.2) {$1$};
    \node[below] at (5,-.2) {$p_\bullet$};
    \node[below] at (15,-.2) {$p_\bullet$};
    \node[below] at (-.4,4) {$\rho$};
    \node[below] at (9.6,4) {$\rho$};
    \draw (2,2) .. controls (1.5,1.5) and (1,1) .. (0,.8) -- (0,0);
    \draw (12,1.7) .. controls (11,1.4) and (10.5,1.2) .. (10,0) -- (10,2.2);
    \node[left] at (2, .5) {$v_\star > 0$};
    \node[left] at (11.9, 1.85) {$v_\star < 0$};
  \end{tikzpicture}
  \caption{Phase diagrams corresponding to lazy and non-lazy particles}
  \label{f:phase_diagram}
\end{figure}
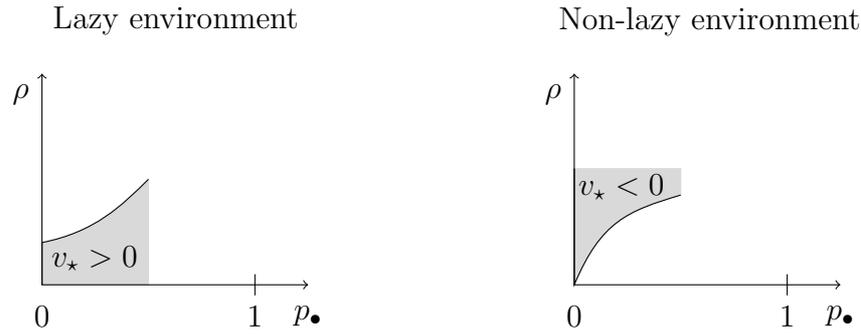

As already mentioned, the ballisticity condition \eqref{e:BAL}
can be used to study further asymptotic properties of the random walker.
The following theorem summarizes new results as well as previous results from \cite{HHSST14}.
\begin{theorem}\label{t:limits_permeable}
Fix $0 \le p_\bullet < p_\circ \le 1$, $\rho \ge 0$, $q_0 \in [0,1)$ and
assume that \eqref{e:BAL} holds with $v_\star \neq 0$.
Assume additionally that
\begin{equation}\label{e:ASSLIM1}
\text{a)} \qquad p_\bullet > 0 \quad \text{ if } \quad v_\star > 0
\end{equation} 
or
\begin{equation}\label{e:ASSLIM2}
\text{b)} \qquad p_\circ <1 \quad \text{ if } \quad v_\star < 0.
\end{equation} 
Then there exist $v = v(p_\circ,p_\bullet, q_0, \rho) \in \R$
and $\sigma = \sigma(p_\circ, p_\bullet, q_0, \rho) \in (0,\infty)$
satisfying $v v_\star >0$, $|v| \le |v_\star|$
and such that the following hold:
\begin{enumerate}
\item (Strong law of large numbers)
\begin{equation}\label{e:LLN_perm}
\lim_{n \to \infty} \frac{X_n}{n} = v \;\;\; \P^\rho \text{-a.s.}
\end{equation}
\item (Functional central limit theorem)
Under $\P^\rho$, the sequence of processes
\begin{equation}\label{e:FCLT_perm}
\left( \frac{X_{\lfloor n t\rfloor} - \lfloor n t \rfloor v}{\sigma \sqrt{n}}\right)_{t \ge 0}, \quad n \in \N,
\end{equation}
converges in distribution as $n \to \infty$ (with respect to the Skorohod topology) to a standard Brownian motion.
\item (Large deviation bounds)
For any $\varepsilon > 0$, there exist constants $c_1, c_2 > 0$ such that
\begin{equation}\label{e:LD_perm}
\P^{\rho} \left( \left|\frac{X_n}{n} - v \right| > \varepsilon \right) \le c_1 e^{-c_2 (\log n)^{\gamma}} \;\;\; \forall \; n \in \N.
\end{equation}
\end{enumerate}
\end{theorem}

\vspace{5pt}

At this point, a few remarks are in order:

\vspace{5pt}
\noindent
{\bf 1.}
Theorems~\ref{t:ballisticity_permeable} and \ref{t:ballisticity_impermeable}
are proved with the help of a renormalization scheme taken from \cite{manyRW}.
In fact, given the setup developed therein, our problem is reduced to
proving two \emph{triggering theorems}, which are key a priori estimates
on the probability of certain undesired events (cf.\ Section~\ref{s:renorma}).
This step is here much more involved than in the high-density regime considered in \cite{manyRW, HHSST14}:
for Theorem~\ref{t:ballisticity_permeable}, it is proved through a careful
analysis of the behaviour of $X$ under decreasing densities and,
for Theorem~\ref{t:ballisticity_impermeable}, by comparison with the front of an infection model
(cf.\ Section~\ref{s:trigger}).

\vspace{5pt}
\noindent
{\bf 2.}
Theorem~\ref{t:limits_permeable} is proved via a regeneration argument as in \cite{HHSST14}.
Note that the assumption $p_\circ > p_\bullet$ implies no loss of generality.
The conditions on $p_\circ, p_\bullet$ in items $a)$ and $b)$ can be seen as \emph{ellipticity assumptions},
as they allow the random walk to take jumps in the direction of $v_\star$ independently of the environment.
Under $b)$, the conclusion already follows from \cite[Theorem~1.4]{HHSST14} (and reflection symmetry); 
in this case, the ellipticity condition can be in fact relaxed using techniques from the proof of \cite[Theorem~5.2]{manyRW}.
The proof of the theorem under $a)$ will be given in Section~\ref{s:reg_lowdensity} below.
The control of the regeneration time is here different,
as the asymmetry in law of occupied/empty sites in the random environment leads to different
monotonicity properties once the roles of $p_\circ$ and $p_\bullet$ are exchanged (cf.\ Section~\ref{ss:prooftail}).
We are presently unable to extend this analysis to the non-elliptic case, i.e., when $p_\bullet=0$.

\vspace{5pt}
\noindent
{\bf 3.}
Under the conditions of Theorems~\ref{t:ballisticity_permeable} and~\ref{t:limits_permeable},
it is possible to show that the speed $v$ in \eqref{e:LLN_perm} above
is a continuous function of $\rho$ in the interval $[0,\hat{\rho}]$ (cf.\ Remark~4.8 of \cite{HHSST14}).
In particular, for fixed $p_\bullet>0$,
$v \to v_\circ$ as $\rho \to 0$.
When $p_\bullet = 0$, we also expect that $v_\star$
in Theorem~\ref{t:ballisticity_permeable} may be taken arbitrarily close to $v_\circ$ by making $\hat{\rho}$ sufficiently small,
but we are currently not able to prove this.

\vspace{5pt}
\noindent
{\bf 4.}
Our results could be presumably extended to higher dimensions and more general transition kernels,
but extra work would be required.
The approach of \cite{manyRW} does not help here,
the problem being again the asymmetry between occupied/empty sites in the environment.
For $2$-state transition kernels,
the approach of \cite{BH14} could be possibly made to work, however several technical steps would need to be adapted.

\vspace{5pt}
\noindent
{\bf 5.}
A crossover from positive to negative speed of a RWDRE is also obtained in \cite{HS15},
where the random environment is a simple symmetric exclusion process.
The transition is observed when varying the jump speed of the exclusion particles.
We also mention \cite{ABF16a}, where very interesting symmetry properties of the speed are obtained
(in particular for the case where the environment is given by the East model).

\vspace{10pt}
The rest of the paper is organized as follows.
A short overview of the literature in our context is provided next
in Section~\ref{ss:connections}.
Technical statements start in Section~\ref{s:construction},
where we provide a convenient construction of our model.
Theorems~\ref{t:ballisticity_permeable}--\ref{t:ballisticity_impermeable}
are proved in Section~\ref{s:renorma} by application of a renormalization setup from \cite{manyRW};
the proof relies on two triggering theorems that are in turn proved in Section~\ref{s:trigger}.
Finally, in Section~\ref{s:reg_lowdensity} we prove Theorem~\ref{t:limits_permeable}
by means of a regeneration argument.


\subsection{Connections to the literature}
\label{ss:connections}

Models of random walks in random environments have been studied since many years.
The setup of the present paper fits in the context of RWDRE in interacting particle systems,
as introduced in \cite{AdHR10, AdHR11}.
One motivation for RWDRE in one dimension comes from the static version (i.e., where the environment is constant in time),
which is known to exhibit, in some regimes, anomalous behaviour such as transience with zero speed \cite{Solomon75} and non-diffusive scalings \cite{KKS75}, in sharp contrast to usual homogeneous random walks. These phenomena are related to \emph{trapping effects},
whereby regions of the lattice with atypical environment configurations tend to hold the random walker for abnormally large times.
Since in the dynamic case the trapping regions may disappear, the question is raised of whether the phenomena remain.
This question is up to now only partially answered in the literature, mostly by identifying regimes with no anomalous behaviour.
For example, \cite{ABF16b, AdHR11, BV16, dHdSS13, RV13} identify general conditions under which laws of large numbers and central limit theorems hold, and \cite{ABF16a, AdSV13, dHdS14, MV15, HS15} study particular examples.
We also mention the works \cite{AFJV15, AdHR10, AJV14, CDRRS13, OS16, dS14}, concerning other asymptotic results.
For further discussion, we refer the reader to \cite{manyRW, HHSST14} and the references therein.

\vspace{10pt}
\noindent
{\bf Acknowledgments.}
OB acknowledges the support of the French Ministry of Education
through the ANR 2010 BLAN 0108 01 grant.
MH was partially supported by CNPq grants 248718/2013-4 and 406659/2016-8 and by ERC AG "COMPASP".
RSdS was supported by the German DFG project KO 2205/13 ``Random mass flow through random potential''.
AT was supported by CNPq grants 306348/2012-8 and 478577/2012-5 and by FAPERJ grant 202.231/2015.
OB, MH and RSdS thank IMPA for hospitality and financial support.
AT and MH thank the CIB for hospitality and financial support.
RSdS thanks ICJ for hospitality and financial support.
The research leading to the present results benefited from
the financial support of the seventh Framework Program of the European Union (7ePC/2007-2013),
grant agreement n\textsuperscript{o}266638.
Part of this work was carried on while MH was on a sabbatical year on the University of Geneva. 
He thanks the mathematics department of this university for the financial support.
OB and AT thank the University of Geneva for hospitality and financial support.

\section{Construction}
\label{s:construction}
In this section, we provide a convenient construction of our random environment and our random walker by means
of a point process of trajectories as in \cite{HHSST14}.

Define the set of doubly-infinite trajectories
\begin{equation}
  \label{e:def_W}
  W = \Big\{ w:\Z \to \Z \colon\, |w(i + 1) - w(i)| \le 1 \;\; \forall \; i \in \Z \Big\}.
\end{equation}
Note that trajectories in $W$ are allowed to jump to the left, jump to the right, or stay put.
We endow the set $W$ with the $\sigma$-algebra $\mathcal{W}$ generated by the canonical coordinates $w \mapsto w(i)$, $i \in \Z$.

Let $(S^{z,i})_{z \in \Z, i \in \N}$ be a collection of independent random elements of $W$,
with each $S^{z,i} = (S^{z,i}_\ell)_{\ell \in \Z}$ distributed as
a double-sided simple symmetric random walk on $\Z$ started at $z$,
i.e., the past $(S^{z,i}_{-\ell})_{\ell \ge 0}$ and future $(S^{z,i}_{\ell})_{\ell \ge 0}$
are i.i.d.\ and distributed as a simple symmetric random walk satisfying \eqref{e:defq_0}.

For a subset $K \subset \Z^2$, denote by $W_K$ the set of trajectories in $W$ that intersect $K$,
i.e., $W_K := \{w \in W \colon\, \exists\, i \in \Z, (w(i),i) \in K\}$.
We define the space of point measures
\begin{equation}
  \label{e:Omega}
  \Omega = \Big\{ \omega = \sum_i \delta_{w_i};\; w_i \in W \text{ and } \omega(W_{\{y\}}) < \infty \text{ for every } y \in \Z^d \times \Z \Big\},
\end{equation}
endowed with the $\sigma$-algebra generated by the evaluation maps $\omega \mapsto \omega(W_K)$, $K \subset \Z^2$.

For a fixed initial configuration $\eta = (\eta(x))_{x \in \Z} \in \Z_+^{\Z}$,
we define the random element
\begin{equation}\label{e:defomega}
\omega := \sum_{z \in \Z} \sum_{i \le \eta(z)} \delta_{S^{z,i}} \in \Omega
\end{equation}
and, for $y \in \mathbb{Z}^2$, we set
\begin{equation}\label{e:defN}
N(y) := \omega(W_{\{y\}}).
\end{equation}

Let $U = (U_y)_{y \in Z}$ be i.i.d.\ Uniform$[0,1]$ random variables
independent of $\omega$.
We define the space-time processes $Y^y = (Y^y_n)_{n \in \Z_+}$, $y \in \Z^2$ by setting
\begin{equation}\label{e:defY}
\begin{aligned}
Y^y_0 & = y,\\
Y^y_{n+1} & = Y^y_n + \left\{
\begin{array}{ll}
\bigl(2 \mathbbm{1}_{\{U_{Y^y_n} \leq p_\circ\}} -1, 1 \bigr) & \text{ if } \;\; N(Y^y_n) = 0,\\
\bigl(2 \mathbbm{1}_{\{U_{Y^y_n} \leq p_\bullet\}} -1,1 \bigr) & \text{ if } \;\; N(Y^y_n) \ge 1,
\end{array}\right.
\quad n \in \Z_+.
\end{aligned}
\end{equation}
For $y = (x,t) \in \Z^2$,
we define the random walkers $X^y = (X^y_n)_{n \in \Z_+}$
by the relation $Y^y_n = (X^y_n, n+t)$, i.e., $X^y_n$ is the spatial projection of $Y^y_n$.
Writing $X = X^0$, one may check that the pair $(N,X)$ has indeed the distribution described in Section \ref{s:intro}.

For $\eta \in \Z_+^{\Z}$ fixed, we denote by $\P_\eta$ the joint law of $\omega$ and $U = (U_y)_{y \in \Z^d \times \Z}$.
For $\rho > 0$, denote by $\nu_\rho$ the product Poisson($\rho$) law on $\Z_+^{\Z}$.
We write $\P^{\rho} = \int \P_\eta \nu_\rho (d \eta)$, i.e., $\P^\rho$ is the joint law of $\omega$ and $U$
when $\eta$ is distributed as $\nu_\rho$.
Our configuration space will be taken as $\overline{\Omega} := \Omega \times [0,1]^{\Z^d \times \Z}$,
equipped with the product $\sigma$-algebra.

An important observation is that, under $\P^\rho$,
$\omega$ is a Poisson point process on $\Omega$
with intensity measure $\rho \mu$, where
\begin{equation}
\label{e:defmu}
  \mu = \sum_{z \in \Z^d} P_z
\end{equation}
and $P_z$ is the law of $S + z$ as an element of $W$.
Note that, under $\P^\rho$, the law of $(\omega, U)$ is invariant with respect to space-time translations;
in particular, the law of $Y^y-y$ does not depend on $y$.

We will need the following definition.
\begin{definition}
\label{d:monotone}
For $\omega,\omega' \in \Omega$, we say that $\omega \leq \omega'$ when $\omega(A) \le \omega'(A)$
for all $A \in \cW$.
We say that a random variable $f\colon\,\overline{\Omega}
\to \R$ is \emph{non-decreasing} when $f(\omega, \xi) \leq f(\omega', \xi)$ for all $\omega
\leq \omega'$ and all $\xi \in [0,1]^{\Z^2}$.
We say that $f$ is \emph{non-increasing} if $-f$ is non-decreasing.
We extend these definitions to events $A$ in $\sigma(\omega,U)$ by considering $f=\mathbbm{1}_A$.
Standard coupling arguments imply that
$\mathbb{E}^{\rho}(f) \leq \mathbb{E}^{\rho'}(f)$ for all non-increasing random variables
$f$ and all $\rho \leq \rho'$.
\end{definition}

\begin{remark}
\label{r:monotone}
The above construction provides two forms of monotonicity:\\
(i) Initial position: If $x \leq x'$ have the same parity (i.e., $x' - x \in 2 \mathbb{Z}$),
then
\begin{equation}
X^{(x,n)}_i \leq X^{(x',n)}_i \qquad \forall\, n \in \mathbb{Z}\,\,\forall\, i \in \mathbb{Z}_+.
\end{equation}
(ii) Environment: If $v_\circ \geq v_\bullet$, then $X^y_n$ is
non-increasing (in the sense of Definition~\ref{d:monotone}) for any $y \in \Z^2$, $n \in \mathbb{Z}_+$.
\end{remark}



\section{Renormalization: proof of Theorems~\ref{t:ballisticity_permeable}--\ref{t:ballisticity_impermeable}}
\label{s:renorma}
In this section, we apply the renormalization setup from Section~3 of \cite{manyRW}
to reduce the proof of our main results to the following two \emph{triggering statements}: 

\begin{theorem}\label{t:triggerperm}
Assume $p_\bullet \vee q_0>0$.
There exists $c = c(p_\circ, p_\bullet, q_0) > 0$ such that
\begin{equation}\label{e:triggerperm}
\P^{L^{-1/16}}(X_L <L^{15/16})\leq c\exp \left\{ -c^{-1}(\log L)^2 \right\} \;\;\; \forall\; L \in \N.
\end{equation}
\end{theorem}

\begin{theorem}\label{t:triggerimperm}
Assume $q_0 = 0$.
For any $\hat{\rho}>0$,
there exist $\hat{v} = \hat{v}(\hat{\rho}) < 0$
and $c >0$ such that the following holds.
For any $\hat{L} \in \N$,
there exists $p_\star = p_\star(\hat{\rho}, p_\circ, \hat{L}) \in (0,1)$
such that, if $p_\bullet \le p_\star$, then
\begin{equation}\label{e:triggerimperm}
\P^{\hat{\rho}} \left( X_{\hat{L}} > \hat{v} \hat{L} \right) \le c \exp \left\{- c^{-1} (\log \hat{L})^{3/2} \right\}.
\end{equation}
\end{theorem}
The proof of Theorems~\ref{t:triggerperm}--\ref{t:triggerimperm} will be given in Section~\ref{s:trigger}.
Next we use \cite[Corollary~3.11]{manyRW} to show how these two theorems respectively imply Theorems~\ref{t:ballisticity_permeable} and \ref{t:ballisticity_impermeable}.

\begin{proof}[Proof of Theorem~\ref{t:ballisticity_permeable}]
Define a local function $g: \overline{\Omega} \to [-1,1]$ by setting
\begin{equation}\label{e:defg}
g (\omega, U) =  \begin{cases}
    1, & \quad \text{if $U_0 < p_\bullet$, or if $N(0) = 0$ and $U_0 < p_\circ$,}\\
    -1, & \quad \text{if $U_0 \geq p_\circ$ or if $N(0) > 0$ and $U_0 \geq p_\bullet$},
  \end{cases}
\end{equation}
i.e., the function $g$ returns the first step of the random walker $X^0$ for a given realization of $\omega, U$.
Then we define a function $H: \overline{\Omega} \times \mathbb{Z} \to \{0, 1\}$ by
\begin{equation}\label{e:defH}
H \big( (\omega, U), z \big) = 1_{\{g(\omega, U)=z\}}.
\end{equation}
In words, $H$ decides whether a jump $z$ is correct ($H=1$) or not ($H=0$) 
for a given realization of $\omega,U$ according to whether
the actual random walk $X^0$ would take $z$ as its first jump or not.
Recall now the definition of a $(0,L,H)$-crossing in the paragraph after equation (3.41) of \cite{manyRW},
and note that
\begin{display}\label{e:crossing}
$\sigma : [0, \infty) \cap \mathbb{Z} \to \mathbb{Z}$ is a $(0, L, H)$-crossing if and only if \\ $\sigma_t = X^y_t$ for every $t \in [0, L) \cap \Z$ and some $y \in \{0, \dots, L\} \times \{0\}$,
\end{display}
i.e., the only $(0,L,H)$-crossings are the trajectories of the RWDRE with initial position in $\{0, \ldots, L\}$.
Recall also the definition of averages along a crossing $\sigma$,
\begin{equation}
\chi^g_\sigma (\omega,U) := \frac{1}{L} \sum_{i=n}^{n + L - 1} g (\theta_{(\sigma(i), i)}(\omega,U)),
\end{equation}
to note the following correspondence between events: for any $L\in \N$, $\hat{v} >0$,
\begin{equation}\label{e:equalityevents}
  \Big\{ \exists \, (0, L, H) \text{-crossing } \sigma \, \colon\, \chi^g_\sigma \leq \hat{v} \Big\} =
  \Big\{ \exists \text{ $x \in \{0, \dots, L-1\} \colon$ $X^{(x,0)}_L - x \leq \hat{v} L$} \Big\}.
\end{equation}

Since, for $v_\star \in (0,1)$,
\begin{equation}\label{e:prthmbalperm1}
\mathbb{P}^{L^{-1/16}} \Big( \text{$ \exists \, n \geq 1 \colon\, X^0_n < v_\star n - L $} \Big)
\leq \mathbb{P}^{L^{-1/16}} \Big( \text{$ \exists \, n \geq L/2 \colon\, X^0_n \leq v_\star n$} \Big),
\end{equation}
we only need to bound the right-hand side for some $v_\star \in (0,1)$.
Now, by \eqref{e:equalityevents}, translation invariance and Theorem~\ref{t:triggerperm}, for all $\hat{L}$ large enough,
\begin{equation}
\begin{split}
\mathbb{P}^{\hat{L}^{-1/16}} & \Big( \exists \text{ a $(0, \hat{L}, H)$-crossing $\sigma$ with $\chi^g_\sigma \leq \hat{L}^{-1/16}$} \Big) \leq \hat{L} \mathbb{P}^{\hat{L}^{-1/16}} \Big( X^0_{\hat{L}} \leq \hat{L}^{15/16} \Big)\\
& \overset{\text{Theorem~\ref{t:triggerperm}}}\leq c  \hat{L} \exp \big\{ -c^{-1} (\log \hat{L})^2 \big\} < \exp(-(\log \hat{L})^{3/2}).
\end{split}
\end{equation}
Noting that the events in \eqref{e:equalityevents} are measurable in $\sigma(N(y), U_y \colon\, y \in B_{0,L})$ 
(where $B_{0,L}:= ([-2L, 3L) \times [0,L)) \cap \Z^2$),
and are non-decreasing by \eqref{e:assumpvel}, we verify the assumptions of Corollary~3.11 in \cite{manyRW}
(taking $v(L) = \rho(L) = L^{-15/16}$, and $\hat{L} = L_{\hat{k}}$ for some $\hat{k}$ large enough), obtaining
$v_\star \in (0,1)$, $\rho_\star > 0$ and $c >0$ such that, for all $\rho \le \rho_\star$,
\begin{equation}
\begin{split}
\mathbb{P}^{\rho}\Big( X^0_n \leq v_\star n \Big) & \leq \P^{\rho} \Big( \exists \text{ a $(0, n, H)$-crossing $\sigma$ with $\chi^g_\sigma \leq v_\star$} \Big)\\
 & \leq c^{-1} \exp \big( -c (\log n)^{3/2} \big)
\end{split}
\end{equation}
for all $n \in \Z_+$.
To conclude, sum over $n \geq L/2$ and apply the union bound to \eqref{e:prthmbalperm1}.
\end{proof}

\begin{proof}[Proof of Theorem~\ref{t:ballisticity_impermeable}]
This time, we define $g:\overline{\Omega} \to [-1,1]$ as
\begin{equation}
g (\omega, U) =  \begin{cases}
    -1, & \quad \text{if $U_0 < p_\bullet \wedge p_\circ$, or if $\omega(W_0) = 0$ and $U_0 < p_\circ$,}\\
    1, & \quad \text{otherwise}.
  \end{cases}
\end{equation}
For $y \in \Z^2$,
define a space-time process $\widetilde{Y}^y_t$, $t \in \Z_+$ by setting, analogously to \eqref{e:defY},
\begin{equation}\label{e:deftildeY}
\widetilde{Y}^y_0 = y \quad \text{ and } \quad \widetilde{Y}^y_{t+1} = \widetilde{Y}^y_t + (g(\theta_{\widetilde{Y}^y_t}(\omega, U)), 1), \;\;\; t \in \Z_+.
\end{equation}
Denote by $\widetilde{X}^y_t$ the first coordinate of $\widetilde{Y}^y_t$.
Note that, by invariance in law of $\omega$ under reflection through the origin, $\widetilde{X}^y$ has the same distribution as $-X^y$.
Setting $H:\overline{\Omega} \times \Z \to \{0,1\}$ as in \eqref{e:defH}, we analogously obtain \eqref{e:crossing}--\eqref{e:equalityevents} with $X$ substituted by $\widetilde{X}$.

Fix now $\gamma \in (1,3/2)$ and take $k_o$ as in Corollary~3.11 of \cite{manyRW}.
Fix $\rho>0$ and consider an auxiliary density $\hat{\rho}>0$, to be fixed later.
For this $\hat{\rho}$, let $\hat{v}<0$ as in Theorem~\ref{t:triggerimperm}; we may assume that $|\hat{v}| < 1$.
Fix $\hat{k} \ge k_o$, $p_\circ \in [0,1]$ and let $p_\star$ be as in Theorem~\ref{t:triggerimperm} for $\hat{L} = L_{\hat{k}}$.
Reasoning as in the proof of Theorem~\ref{t:ballisticity_permeable}, we see that, if $p_\bullet \le p_\star$, then
\begin{equation}
\begin{split}
\P^{\hat{\rho}} & \Big( \exists \text{ a $(0, \hat{L}, H)$-crossing $\sigma$ with $\chi^g_\sigma \leq |\hat{v}|$} \Big) \leq \hat{L} \P^{\hat{\rho}} \Big( X^0_{\hat{L}} \geq \hat{L} \hat{v} \Big)\\
& \overset{\text{Theorem~\ref{t:triggerimperm}}}\leq c \hat{L} \exp \big\{ -c^{-1} (\log \hat{L})^{3/2} \big\} < \exp(-(\log \hat{L})^{\gamma})
\end{split}
\end{equation}
whenever $\hat{k}$ (and thus $\hat{L}$) is large enough.
The events in \eqref{e:equalityevents} (with $X$ replaced by $\widetilde{X}$) are again measurable in $\sigma(N(y), U_y \colon y \in B_{0,L})$,
and are either always non-decreasing, or always non-increasing (depending on whether $p_\circ \ge p_\bullet$ or not).
Applying \cite[Corollary~3.11]{manyRW} (with $v(L) = |\hat{v}|$, $\rho(L) = \hat{\rho}$) we obtain $\rho_\infty, c>0$ depending on $\hat{\rho}$ such that
\begin{equation}\label{e:prbalimperm}
\begin{split}
\mathbb{P}^{\rho_\infty}\Big( X^0_n \geq \hat{v} n \Big) \leq c^{-1} \exp \big( -c (\log n)^{\gamma} \big)
\end{split}
\end{equation}
for all $n \in \Z_+$.
Now we note that, using the explicit expression for $\rho_\infty$
mentioned in the proof of \cite[Corollary~3.11]{manyRW}, 
we may choose $\hat{\rho}$ in such a way that \eqref{e:prbalimperm} is still valid with $\rho$ in place of $\rho_\infty$.
To conclude, sum \eqref{e:prbalimperm} over $n \ge L/2$ and use 
$\{\exists\, n \ge 1 \colon X^0_n > \hat{v}n + L \} \subset \{ \exists\, n \ge L/2 \colon X^0_n \geq \hat{v} n\}$ together with a union bound.
\end{proof}



\section{Triggering: proof of Theorems~\ref{t:triggerperm}--\ref{t:triggerimperm}}
\label{s:trigger}
Here we give the proofs of Theorem~\ref{t:triggerperm} (Section~\ref{ss:trigger_perm})
and Theorem~\ref{t:triggerimperm} (Section~\ref{ss:trigger_imperm}).
\subsection{Permeable systems at low density}
\label{ss:trigger_perm}
Throughout this section, we assume $p_\bullet \vee q_0 > 0$ (and $v_\circ > 0 \ge v_\bullet)$.
As mentioned in the introduction, 
we call this case permeable since the random walker is able to cross over particles of the environment.
The usefulness of this condition comes from the fact that $X$ may be coupled
with an independent homogeneous random walk $\bar{X}$ with drift $v_\circ$ (which we call a ``ghost walker'')
such that, whenever the initial configuration $\eta$ consists of at most one particle that is not at the origin,
there is a positive probability that $X_n = \bar{X}_n$ for all $n \in \Z_+$.
In fact, we will show that this probability decays at most exponentially in the number of particles of the environment.
This suggests the following strategy:
whenever a ``ghost walker'' is started to the left of $X$, 
it can ``push'' $X$ to the right.
This may happen with small probability but, if enough time is given, many trials are possible
and so there is a large probability that at least one of them succeeds.

In order to implement this idea, 
we work first in a time scale at which typical empty regions in the initial configuration remain empty, 
and the number of particles between such regions is relatively small. 
This ensures that $X$ does not move very far to the left,
and that the ``ghost walkers'' do not meet too many particles on their way.
The original scale is then reached via translation-invariance and a union bound.

We proceed to formalize the strategy outlined above.
In the following, we state two propositions which will then be used to prove Theorem~\ref{t:triggerperm}.
Their proofs are postponed to Sections~\ref{ss:proofproplowerestimate}--\ref{ss:proofpropcrossingtraps} below.

First of all we define the ghost walkers.
For $(x,t) \in \Z^2$, put
\begin{equation}\label{defbarXxt}
\begin{array}{lcl}
\bar{X}^{(x,t)}_0 & := & x,\\
\bar{X}^{(x,t)}_{s+1} & := & \bar{X}^{(x,t)}_{s} 
+ \left\{\begin{array}{ll} 1 & \text{ if } U_{(\bar{X}^{(x,t)}_s,s+t)} \le p_\circ,\\
-1 & \text{ otherwise.}
\end{array}\right.
\quad s  \in \Z_+.
\end{array}
\end{equation}
Then $\bar{X}^{(x,t)}$ is a simple random walk with drift $v_\circ$ started at $x$.
For $T \in [0,\infty]$, let
\begin{equation}\label{defGinfty}
G^{(x,t)}_{T} := \left\{X_s^{(x,t)}=\bar{X}_s^{(x,t)}\ \forall s\in [0, T]\right\}
\end{equation}
be the good event where the random walk $X^{(x,t)}$ follows $\bar{X}^{(x,t)}$ up to time $T$.
A comparison between $X$ and $\bar{X}^{(x,t)}$ on this event is given by the next lemma.
\begin{lemma}\label{l:compbarX}
Fix $(x,t) \in \Z^2$ with $x \in 2 \Z$.
If $X_t \ge x$ and $G^{(x,t)}_T$ occurs, then 
\[ X_{t+s} \ge \bar{X}^{(x,t)}_s \text{ for all } s \in [0, T].\]
\end{lemma}
\begin{proof}
Follows from Remark~\ref{r:monotone}(i) and the definitions of $X$, $\bar{X}$, $G^{(x,t)}_T$.
\end{proof}

To set up the scales for our proof, we fix $\alpha, \beta, \beta' \in (0,1)$ satisfying
\begin{equation}\label{e:relationscales}
0 < \frac{\alpha}{2} < \beta' < \beta < \alpha < 2 \beta < \frac18
\end{equation}
and we let
\begin{eqnarray}
T_i &:=& i 2 \lfloor 2 v_\circ^{-1} L^\beta \rfloor, \;\; i \in [0, M_L] \cap \Z \;\;\;\; \text{ where } \;\; M_L := \frac14 v_\circ L^{\alpha-\beta},\label{e:defT}\\ 
\ell_L &:=& \lfloor L^{\beta'} \rfloor. \label{e:defellL}
\end{eqnarray}
We assume that $L$ is large enough so that $\ell_L, M_L \ge 1$.

If $p_\bullet = 0$, it is not possible to couple $X^{(x,t)}_1$ and $\bar{X}^{(x,t)}_1$
if there is a particle at $(x,t)$.
Thus, if we aim to control $G^{(x,t)}_T$, we should have $N(x,t)=0$.
To that end, define
\begin{equation}\label{e:defhatz}
\hat{Z} := \max \left\{ z < -2 \ell_L \colon\, N(x,0) = 0 \; \forall \; x \in \Z, |x - z| \le 2 \ell_L \right\}
\end{equation}
to be the center of the first interval of $4 \ell_L + 1$ empty sites to the left of the origin in the initial configuration.
Then set
\begin{equation}\label{e:defx-}
X_- := \left\{
\begin{array}{ll}
\hat{Z} - \ell_L & \text{ if } \hat{Z} - \ell_L \in 2 \Z, \\
\hat{Z} - \ell_L + 1 & \text{ otherwise.}
\end{array} \right.
\end{equation}
Note that $X_- \in 2 \Z$.

In order to use Lemma~\ref{l:compbarX}, 
we must control the probability that $X$ crosses $X_-$ before time $L^\alpha$.
This is the content of the following proposition,
whose proof relies on standard properties of simple random walks and Poisson random variables.
\begin{proposition}\label{p:lowerestimate}
There exist $c,\varepsilon > 0$ such that, for all large enough $L \in \N$,
\begin{equation}\label{e:lowerestimate}
\P^{L^{-\frac{1}{16}}} \left( \min_{0 \le s \le L^\alpha} X_s < X_- \right) \le c e^{-c^{-1} L^{\varepsilon}}.
\end{equation}
\end{proposition}

The next proposition shows that, with large probability, one of the $G^{(X_-, T_i)}_{T_1}$'s occurs. 
Its proof depends crucially on the permeability of the system.
\begin{proposition}\label{p:crossingtraps}
There exists $c>0$ such that, for all large enough $L \in \N$,
\begin{equation}\label{e:crossingtraps}
\P^{L^{-\frac{1}{16}}} \left( \bigcup_{i \in [0,M_L-1]} G^{(X_-, T_i)}_{T_1} \cap \{\bar{X}^{(X_-, T_i)}_{T_1} \ge L^\beta \}  \right) \ge 1 - ce^{-c^{-1}(\log L)^2}.
\end{equation}
\end{proposition}

We are now ready to prove Theorem~\ref{t:triggerperm}.
\begin{proof}[Proof of Theorem~\ref{t:triggerperm}]
First we argue that, for some constant $c > 0$,
\begin{equation}\label{pptrigger0}
\P^{L^{-\frac{1}{16}}}\left( \sup_{0 \le s \le L^\alpha} X_s < L^\beta \right) \le c e^{-c^{-1}(\log L)^2} \;\;\;\forall \; L \in \N.
\end{equation}
Indeed, by Lemma~\ref{l:compbarX}, the complement of the event in \eqref{pptrigger0} contains the event
\[
\left\{\min_{0 \le s \le L^\alpha} X_s \ge X_- \right\} \bigcap \bigcup_{i \in [0, M_L-1]} G^{(X_-, T_i)}_{T_1} \cap \{\bar{X}^{(X_-, T_i)}_{T_1} \ge L^\beta \},
\]
which by Propositions~\ref{p:lowerestimate}--\ref{p:crossingtraps} has probability at least $1 - c e^{-c^{-1}(\log L)^2}$.

Now let $\sigma_k$ be the sequence of random times when the increments of $X$ are at least $L^\beta$, i.e.,
$\sigma_0 := 0$ and recursively
\begin{equation}\label{defsigma}
\sigma_{k+1} := \inf\{ s > \sigma_k \colon\, X_s - X_{\sigma_k} \ge L^\beta \}, \;\; k \ge 0.
\end{equation}
Setting $K := \sup\{k \ge 0 \colon\, \sigma_k \le L \}$, we obtain
\begin{equation}\label{pptrigger1}
X_L = \sum_{i=0}^{K-1} X_{\sigma_{i+1}}-X_{\sigma_i} + X_L - X_{\sigma_{K}} \ge K L^\beta - (\sigma_{K+1} - \sigma_{K}).
\end{equation}
On the event
\begin{equation}\label{defAL}
B_L := \{\sigma_{k+1} - \sigma_k \le L^{\alpha} \;\forall\; k = 0, \ldots, K\},
\end{equation}
we have $K \ge L^{1-\alpha} -1$. Therefore, by \eqref{pptrigger1}, on $B_L$ we have
\begin{equation}\label{pptrigger2}
X_L \ge L^{1-\alpha + \beta} - L^\beta - L^\alpha \ge L^{\frac{15}{16}}
\end{equation}
for large $L$ since $1 - \alpha + \beta > 15/16 > \alpha > \beta$.
Thus we only need to control the probability of $B_L$.
But, by the definition of $X$,
\begin{align}\label{pptrigger3}
\P^{L^{-\frac{1}{16}}}\left( B_L^c \right) 
& \le \P^{L^{-\frac{1}{16}}}\left( \exists\; (x,t) \in [-L, L] \times [0,L] \colon\, \sup_{s \in [0,L^\alpha]}X^{(x,t)}_s < L^\beta \right) \nonumber\\
& \le c L^2 \, \P^{L^{-\frac{1}{16}}}\left( \sup_{0 \le s \le L^\alpha} X_s < L^\beta \right) \le c e^{-c^{-1}(\log L)^2},
\end{align}
where we used a union bound, translation-invariance and \eqref{pptrigger0}.
This completes the proof of Theorem~\ref{t:triggerperm}.
\end{proof}

\subsubsection{Proof of Proposition~\ref{p:lowerestimate}}
\label{ss:proofproplowerestimate}

Recall the definition of $\hat{Z}$ in \eqref{e:defhatz}.
The idea behind the proof of Proposition~\ref{p:lowerestimate} is that, 
with our choice of scales, the interval $[\hat{Z} - \ell_L, \hat{Z} + \ell_L]$
remains empty throughout the time interval $[0,L^\alpha]$.
Since inside this interval $X$ behaves as a random walk with a positive drift,
it avoids $X_- \le \hat{Z} - \ell_L+1$ with large probability.

We first show that $\hat{Z} - 2\ell_L \ge -L^\beta$ with large probability.
\begin{lemma}\label{l:estimateE-}
\begin{equation}\label{e:estimateE-}
\P^{L^{-\frac{1}{16}}} \left( \hat{Z} - 2 \ell_L < - L^\beta \right) \le c e^{-c^{-1}L^{\beta-\beta'}}.
\end{equation}
\end{lemma}
\begin{proof}
We may assume that $L$ is large enough.
Let $E_0 := 0$ and recursively
\begin{equation}\label{defEk}
E_{k+1} := \max \{z < E_k \colon\, N_0(z) > 0 \}, \;\; k \ge 0.
\end{equation}
Then $(E_k-E_{k+1})_{k \ge 0}$ are i.i.d.\ Geom($1 - e^{-L^{-\frac{1}{16}}}$) random variables.
Let 
\begin{equation}\label{defK}
K := \inf\{k \ge 0 \colon\, |E_{k+1}-E_{k}| > 4 \ell_L \}.
\end{equation} 
Then $K+1$ has a geometric distribution with parameter $e^{-4 \ell_L L^{1/16}}$. 
Thus
\begin{align}\label{e:estimK}
\P^{L^{-\frac{1}{16}}} \left( K + 1 > \frac14 L^{\beta-\beta'} \right) & \le (1-e^{-4L^{-(1/16 - \beta')}})^{\frac14L^{\beta-\beta'}} \nonumber\\
& \le 4^{\frac14 L^{\beta-\beta'}} e^{ - \frac14 (1/16-\beta')L^{\beta-\beta'} \log L} \le c e^{-c^{-1}L^{\beta-\beta'}}.
\end{align}
Since $|\hat{Z} - 2 \ell_L| \le 4 \ell_L (K+1)$,
\begin{align}\label{e:estimE+}
\P^{L^{-\frac{1}{16}}} \left( \hat{Z} - 2\ell_L < -L^\beta \right) 
& \le \P^{L^{-\frac{1}{16}}} \left( K + 1> \frac14 L^{\beta-\beta'} \right) \le c e^{-c^{-1}L^{\beta-\beta'}}
\end{align}
by \eqref{e:estimK}. This finishes the proof.
\end{proof}

Next we show that, with large probability, the particles of the random environment
do not penetrate deep inside the empty region up to time $L^\alpha$.
Let
\begin{equation}\label{defcB1}
\mathcal{E}_L := \{N(y) = 0 \;\forall\; y \in [\hat{Z} - \ell_L, \hat{Z} + \ell_L] \times [0,L^\alpha]\}.
\end{equation}
\begin{lemma}\label{l:noparticles}
There exists $c > 0$ such that
\begin{equation}\label{e:noparticles}
\P^{L^{-\frac{1}{16}}} \left( \mathcal{E}_L^c \right) \le c e^{- \frac{1}{c}L^{(\beta - \beta') \wedge (2 \beta' -\alpha)}}.
\end{equation}
\end{lemma}
\begin{proof}
For $x \in \Z$, the random variable
\begin{equation}\label{defNLx}
\widehat{N}_L(x) := \sum_{z \notin [x- 2 \ell_L, x+ 2\ell_L]} \; \sum_{i \le N(z,0)} \mathbf{1}_{\{\exists\; s\in [0,L^\alpha] \colon\, S^{z,i}_s \in [x-\ell_L, x+\ell_L]\}}
\end{equation}
has a Poisson distribution with parameter
\begin{equation}\label{parameterNLx}
\lambda_L(x) := L^{-\frac{1}{16}} \sum_{z \notin [x- 2 \ell_L, x+ 2\ell_L]} P( \exists\; s \in [0,L^\alpha] \colon\, S^{z,1} \in [x-\ell_L, x+ \ell_L]  ),
\end{equation}
where $S^{z,1}$ is a simple symmetric random walk started at $z$ as defined in the introduction.
By standard random walk estimates, we have
\begin{align}\label{e:estimparamNLx}
\lambda_L(x) 
 \le 2 \sum_{k > \ell_L} P \left(\sup_{s \in [0,L^\alpha]} |S^{0,1}_s| \ge k \right) \le c \sum_{k > \lfloor L^{\beta'} \rfloor} e^{-\frac{k^2}{c L^\alpha}} 
 \le c L^{\alpha} e^{-c^{-1}L^{2\beta'-\alpha}}.
\end{align}
Therefore, by Lemma~\ref{l:estimateE-} and \eqref{e:estimparamNLx},
\begin{align}\label{e:nopart1}
\P^{L^{-\frac{1}{16}}} \left( \widehat{N}_L(\hat{Z}) > 0  \right)  
& \le \P^{L^{-\frac{1}{16}}} \left( \hat{Z} < -L^\beta\right)
+ \P^{L^{-\frac{1}{16}}} \left( \exists\; x \in [-L^\beta,0] \colon\, \widehat{N}_L(x) > 0 \right) \nonumber\\
& \le c e^{-c^{-1}L^{(\beta- \beta')}} +  c L^\beta \sup_x \lambda_L(x) \le c e^{-\frac{1}{c}L^{(\beta-\beta')\wedge(2\beta'-\alpha)}}.
\end{align}
Since $N(z,0)=0$ for all $z \in [\hat{Z}-2\ell_L, \hat{Z} + 2\ell]$ by definition,
$\widehat{N}_L(\hat{Z})$ is equal to the total number of particles that enter $[\hat{Z} - \ell_L, \hat{Z}+\ell_L] \times [0,L^\alpha]$.
This completes the proof.
\end{proof}

Let now, for $t \in \N$,
\begin{equation}\label{defhittimes}
\begin{array}{lcl}
H^{(t)}_{+} & := & \inf\{ s \ge 0 \colon\; X^{(\hat{Z},t)}_s -\hat{Z} = \ell_L \}, \\
H^{(t)}_{-} & := & \inf\{ s \ge 0 \colon\; X^{(\hat{Z},t)}_s -\hat{Z} = - \ell_L +1 \}
\end{array}
\end{equation}
be the times when the random walk $X^{(\hat{Z},t)}$ hits the sites $\hat{Z} + \ell_L$ or $\hat{Z} - \ell_L +1$.
Let
\begin{equation}\label{defcDL}
\mathcal{D}_L := \{ H^{(t)}_- > H^{(t)}_+ \wedge (L^{\alpha}-t) \;\forall\; t \in [0,L^\alpha]\}.
\end{equation}
The last lemma of this section shows that also $\mathcal{D}_L$ has large probability.

\begin{lemma}\label{l:blockevent}
\begin{equation}\label{e:blockevent}
\P^{L^{-\frac{1}{16}}} \left( \mathcal{D}_L^c \,\middle|\, \mathcal{E}_L \right) \le c e^{-c^{-1}L^{\beta'}}.
\end{equation}
\end{lemma}
\begin{proof}
Fix $t \in [0,L^\alpha]$ and note that, on the event $\mathcal{E}_L$,
$X^{(\hat{Z},t)}_s-\hat{Z}$ is up to time $H^{(t)}_+ \wedge H^{(t)}_- \wedge (L^\alpha-t)$  equal to $\bar{X}^{(\hat{Z},t)}_s-\hat{Z}$.
The latter is a random walk with drift $v_\circ>0$, so by standard estimates we obtain
\begin{align}\label{pblockevent1}
\P^{L^{-\frac{1}{16}}} \left( H^{(t)}_- \le H^{(t)}_+ \wedge (L^\alpha-t) \,\middle|\, \mathcal{E}_L \right)
& \le \P^{L^{-\frac{1}{16}}} \left( \inf_{s \ge 0} \bar{X}^{(\hat{Z},t)}_s - \hat{Z} \le - \ell_L + 1 \right)  \nonumber\\
& \le ce^{-c^{-1}\ell_L} \le c e^{-c^{-1}L^{\beta'}}.
\end{align}
The proof is completed using \eqref{pblockevent1} and a union bound over $t \in [0,L^\alpha]$.
\end{proof}

With Lemmas~\ref{l:estimateE-}--\ref{l:blockevent} at hand we can finish the proof of Proposition~\ref{p:lowerestimate}.
\begin{proof}[Proof of Proposition~\ref{p:lowerestimate}]
By Lemmas~\ref{l:noparticles}--\ref{l:blockevent},
\begin{equation}\label{e:pproplowerestimate}
\P^{L^{-\frac{1}{16}}} \left( \mathcal{D}_L \right) \ge 1-c e^{-c^{-1}L^{\varepsilon}}
\end{equation}
where $\varepsilon := \beta' \wedge (\beta - \beta') \wedge (2\beta' - \alpha)$.
The proof is finished by noting that, 
since $X$ must hit $\hat{Z}$ in order to reach $\hat{Z} - \ell_L+1 \ge X_-$,
if $\mathcal{D}_L$ occurs then $X_s \ge X_-$ $\forall$ $s \in [0,L^\alpha]$.
\end{proof}

\subsubsection{Proof of Proposition~\ref{p:crossingtraps}}
\label{ss:proofpropcrossingtraps}

The proof of Proposition~\ref{p:crossingtraps} follows two steps that are presented in Lemmas~\ref{l:permeability} and \ref{l:crossfinitetraps}.
We first show an lower bound on the probability of $G_\infty \cap \Lambda_\infty$.
This lower bound is provided in Lemma~\ref{l:permeability} and decays exponentially in the number of particles in $\eta$.
Intuitively speaking this can be interpreted as if the walker had to pay a constant price to ignore each particle.

Then in Lemma~\ref{l:crossfinitetraps} we show that, if the initial configuration has a logarithmic number of particles and we are given enough attempts, the walker is very likely to ignore all of them.

For $(x,t) \in \Z^2$ and $T \in [0,\infty]$, let
\begin{equation}\label{defLambdaxt}
\Lambda^{(x,t)}_T := \left\{ \bar{X}^{(x,t)}_s - x \ge \frac12 v_\circ s \,\;\forall\, s \in [0, T] \right\}.
\end{equation}
When $(x,t) = (0,0)$, we will omit it from the notation for both $G^{(x,t)}_T$ and $\Lambda^{(x,t)}_T$.

For $\eta \in \Z_+^\Z$, denote by
\begin{equation}\label{d:numbereta}
|\eta| := \sum_{z \in \Z} \eta(z) \;\in [0,\infty]
\end{equation}
the total number of particles in $\eta$.
Note that $|N(\cdot,t)| = |\eta|$ a.s.\ under $\P_\eta$.

The first goal of the section is the following key lemma, 
providing a lower bound on the probability of $G_\infty \cap \Lambda_\infty$ when $|\eta| < \infty$ and $\eta(0) = 0$.
\begin{lemma}\label{l:permeability}
There exists $p_* \in (0,1)$ such that
\begin{equation}\label{e:permeability}
\inf_{\eta \colon |\eta| \le k, \eta(0) = 0} \P_\eta \left( G_\infty \cap \Lambda_\infty \right) \ge p_*^{k} \;\;\;\; \forall\; k \ge 0.
\end{equation}
\end{lemma}
In order to prove Lemma~\ref{l:permeability}, we will need an auxiliary result.
For a set $B \subset \Z$ and two configurations 
$\eta, \xi \in \Z_+^\Z$ satisfying $\xi \le \eta$ (i.e., $\xi(x) \le \eta(x) \;\forall\; x \in \Z$), 
let
\begin{equation}\label{defetaB}
\eta^{B,\xi}(x) := \left\{ \begin{array}{ll}
                     \eta(x) - \xi(x) & \text{ if } x \in B, \\
                     \eta(x)  & \text{ otherwise.}
                   \end{array} \right.
\end{equation}
For $A \subset \Z^2$, we write $N(A) = (N(y))_{y \in A}$ and $U_A = (U_y)_{y \in A}$.
The following lemma is a consequence of the i.i.d.\ nature of the particles in the environment.
\begin{lemma}\label{l:localeventslocalparticles}
Let $A \subset \Z^2$ and $B \subset \Z$.
For any two configurations $\xi \le \eta \in \Z_+^\Z$ 
and any measurable bounded function $f$,
\begin{equation}\label{e:localeventslocalparticles}
\begin{aligned}
& \E_{\eta} \left[ f\left(N(A), U_A\right) \,\middle|\, (S^{z,i})_{i \le \xi(z), z \in B}\right] 
 = \E_{\eta^{B,\xi}} \left[ f\left(N(A), U_A \right) \right] \\ 
 & \text{a.s.\ on the event } \{ S^{z,i}_n \cap A = \emptyset \,\;\forall\; n \in \Z, i \le \xi(z), z \in B \}. 
\end{aligned}
\end{equation}
\end{lemma}
\begin{proof}
For $(x,t) \in \Z^2$, let
\begin{equation}\label{defNB}
N^{B,\xi}(x,t) := \sum_{z \notin B} \; \sum_{1 \le i \le \eta(z)} \mathbf{1}_{\{S^{z,i}_t = x\}}
+ \sum_{z \in B} \; \sum_{ \xi(z) < i \le \eta(z)} \mathbf{1}_{\{S^{z,i}_t = x\}}.
\end{equation}
On the event in the second line of \eqref{e:localeventslocalparticles}, $f(N(A),U_A) = f(N^{B,\xi}(A), U_A)$
and the latter is independent of $(S^{(z,i)})_{i \le \xi(z), z \in B}$.
To conclude, note that $N^{B,\xi}$ has under $\P_\eta$ the same distribution
of $N$ under $\P_{\eta^{B,\xi}}$.
\end{proof}

We can now give the proof of Lemma~\ref{l:permeability}.
\begin{proof}[Proof of Lemma~\ref{l:permeability}]
We start with the case $q_0>0$.
We claim that one may assume $\eta(z) = 0$ for all $z \le 0$.
Indeed, apply Lemma~\ref{l:localeventslocalparticles}  with 
$A = \{(x,t) \in \Z^2_+ \colon\, x \ge \tfrac12 v_\circ t\}$, $B= (-\infty, -1] \cap \Z$
and $\xi(z) = \eta(z) \mathbbm{1}_{\{z < 0\}}$ to obtain
\[
\P_\eta(G_\infty \cap \Lambda_\infty) \ge P(S^{0,1}_n -1 \notin A \;\forall\, n \in \Z_+)^{|\xi|} \, \P_{\eta^{B,\xi}} (G_\infty \cap \Lambda_\infty )
\]
where $\eta^{B,\xi}(z) = 0$ for all $z \le 0$ and $|\eta^{B,\xi}| = |\eta|-|\xi|$.
We thus let
\begin{equation}\label{prlperm1}
p_k := \inf_{|\eta|=k, \eta(z) = 0 \,\forall\, z \le 0} \P_\eta(G_\infty \cap \Lambda_\infty).
\end{equation}
It is clear that
\begin{equation}\label{e:inductionbase}
p_0 = \P_0(\Lambda_\infty) = P \left( \bar{X}_n \ge \tfrac12 v_\circ n \;\forall\, n \in \Z_+ \right) > 0.
\end{equation}
Let $A'=\left( \bigcup_{i=0}^2\{(i,i)\} \right) \cup \{(x,t) \in \Z^2 \colon t \ge 3, x \ge \tfrac12 v_\circ t\}$ and 
$B=\{1,2\}$. We say that ``$S^{z,i}$ avoids $A'$'' if $S^{z,i}_n \notin A'$ for all $n \in \Z$. 
Since $q_0>0$,
\begin{equation}\label{prlperm3}
\tilde{p} := \inf_{z \in B} P(S^{z,1} \text{ avoids } A') > 0.
\end{equation}
We will prove that, for all $k \ge 0$,
\begin{equation}\label{prlperm2}
p_k \ge p_{**}^k \quad \text{ where } p_{**} := p_0 \tilde{p}
\end{equation}
by induction on $k$.
Let $|\eta| \ge 1$, $\eta(z) = 0$ for all $z \le 0$, and 
assume that \eqref{prlperm2} has been shown for all $k < |\eta|$.

Assume first that $\eta(1)+\eta(2) \ge 1$ and put $\xi(z) = \eta(z) \mathbbm{1}_{\{1,2\}}(z)$.
Noting that $G_\infty \cap \Lambda_\infty$ is measurable in $\sigma(N(A'), U_{A'})$,
use Lemma~\ref{l:localeventslocalparticles} and the induction hypothesis to write
\begin{align}\label{prlperm4}
\P_\eta (G_\infty, \Lambda_\infty) 
& \ge 
\E_\eta \left[ \prod_{z \in B, i\le \xi(z)} \mathbbm{1}_{\{S^{z,i} \text{ avoids } A'\}} \P_{\eta}
\left(G_\infty \cap \Lambda_\infty \,\middle|\, (S^{z,i})_{z \in B, i \le \xi(z)} \right)\right] \nonumber\\
& \ge \tilde{p}^{|\xi|} p_{|\eta| - |\xi|} \ge \tilde{p} p_{**}^{|\eta|-1} \ge p_{**}^{|\eta|}.
\end{align}

If $\eta(1) + \eta(2) = 0$, let
\begin{equation}\label{e:deftau}
\tau := \inf\{n \in \N \colon N(\bar{X}_n + 1, n) + N(\bar{X}_n+2,n) \ge 1\}.
\end{equation}
Note that $\tau < \infty$ a.s.\ since $\bar{X}$ has a positive drift while the environment particles are symmetric.
Let $\bar{\eta}_\tau(x) = N(\bar{X}_\tau + x, \tau)$ and note that,
since the random walks are all $1$-Lipschitz,
$\bar{\eta}_\tau(z) = 0$ for all $z \le 0$.
Furthermore, $X$ is equal to $\bar{X}$ until time $\tau$ since it meets no environment particles up to this time.
Thus, using the Markov property and \eqref{prlperm4} we can write
\begin{align}\label{prlperm5}
\P_\eta(G_\infty \cap \Lambda_\infty) 
& \ge \P_{\eta} \left( \Lambda_\tau \cap G^{(\bar{X}_\tau, \tau)}_\infty \cap \Lambda^{(\bar{X}_\tau, \tau)} \right) 
\nonumber\\
&= \E_\eta \left[\mathbbm{1}_{\Lambda_\tau} \P_{\bar{\eta}_\tau} \left(G_{\infty} \cap \Lambda_\infty \right)\right] \nonumber\\
& \ge \tilde{p} p_{**}^{|\eta|-1} \P_\eta \left(\Lambda_\tau \right) \ge p_{**}^{|\eta|},
\end{align}
completing the induction step.

We turn now to the case $q_0=0$, $p_\bullet > 0$.
In this case, we can actually control
\begin{equation}\label{prlperm6}
p_k := \inf_{|\eta|=k} \, \P_\eta(G_\infty \cap \Lambda_\infty) = \inf_{y \in \Z^2} \, \inf_{|\eta|=k} \, \P_\eta(G^y_\infty \cap \Lambda^y_\infty) ,
\end{equation}
where the second equality holds by the Markov property, particle conservation and translation invariance.
Let $p_{**} := p_\bullet p_0 \hat{p}$
where $p_0$ is as in \eqref{e:inductionbase}
and
\begin{equation}\label{prlperm7}
\hat{p} := P(S^{0,1} \text{ avoids } A''), \quad A'':= \{(x,t) \in \Z^2 \colon\, t \ge 1, x \ge \tfrac12 v_\circ t\}.
\end{equation}
Then we can prove \eqref{prlperm2} by induction in a similar way as for the previous case.

Indeed, suppose first that $\eta(0) > 0$.
Note that, since $X_1 = 1$ when $U_0 \le p_\bullet$,
\begin{align}\label{prlperm8}
\P_\eta \left( G_\infty \cap \Lambda_\infty \right) 
& \ge \P_\eta \left(U_0 \le p_\bullet, G^{(1,1)}_\infty \cap \Lambda^{(1,1)}_\infty \right) 
= p_\bullet \P_\eta \left( G^{(1,1)}_\infty \cap \Lambda^{(1,1)}_\infty \right) \nonumber\\
& \ge p_\bullet \E_\eta \left[ \prod_{i \le \eta(0)} \mathbbm{1}_{\{S^{0,i} \text{ avoids } A''\}} \P_\eta \left( G^{(1,1)}_\infty \cap \Lambda^{(1,1)}_\infty \,\middle|\, (S^{0,i})_{i \le \eta(0)} \right) \right].
\end{align}
Noting that $G^{(1,1)}_\infty \cap \Lambda^{(1,1)}_\infty$ is measurable in $\sigma(N(A''), U_{A''})$,
we may apply Lemma~\ref{l:localeventslocalparticles} with $B=\{0\}$, $\xi = \eta \mathbbm{1}_0$
followed by the induction hypothesis to obtain
\begin{equation}\label{prlperm9}
\P_\eta \left( G_\infty \cap \Lambda_\infty \right) \ge p_\bullet \hat{p}^{|\eta(0)|} p_{|\eta| - |\eta(0)|} \ge p_\bullet \hat{p} p_{**}^{|\eta|-1}.
\end{equation}

If $\eta(0)=0$, define
\begin{equation}\label{prlperm10}
\tau := \inf \{n \in \N \colon\, N(\bar{X}_n,n) \ge 1\} \in [1,\infty].
\end{equation}
Setting $\bar{\eta}_\tau(x) = N(\bar{X_\tau} + x, \tau)$, use the Markov property and \eqref{prlperm9} to write
\begin{align}\label{prlperm11}
\P_\eta \left(\tau < \infty, G_\infty \cap \Lambda_\infty \right)
& \ge \E_\eta \left[ \mathbbm{1}_{\{\tau < \infty\}} \mathbbm{1}_{\Lambda_\tau} \P_{\eta_\tau} \left( G_\infty \cap \Lambda_\infty \right)\right] \nonumber\\
& \ge p_\bullet \hat{p} p_{**}^{|\eta|-1} \P_\eta \left(\tau < \infty, \Lambda_\tau \right).
\end{align}
Now note that $G_\infty$ occurs if $\tau = \infty$ and use \eqref{prlperm11} to obtain
\begin{align}\label{prlperm12}
\P_\eta \left(G_\infty \cap \Lambda_\infty \right) 
& = \P_\eta(\tau = \infty, \Lambda_\infty) + \P_\eta (\tau < \infty, G_\infty \cap \Lambda_\infty) \nonumber\\
& \ge p_\bullet \hat{p} p_{**}^{|\eta|-1} \left\{\P_\eta(\tau = \infty, \Lambda_\infty) + \P_\eta(\tau < \infty, \Lambda_\infty)\right\} = p_{**}^{|\eta|},
\end{align}
concluding the proof.
\end{proof}

Next we use Lemma~\ref{l:permeability} to show that, if $|\eta|$ is sufficiently small 
and is empty in an interval of radius $\ell_L$ around $0$,
then one of the $G^{(0,T_i)}_{T_1}$'s occurs with large probability.
\begin{lemma}\label{l:crossfinitetraps}
There exist $\delta, \varepsilon, c > 0$ such that
\begin{equation}\label{e:crossfinitetraps}
\inf_{\substack{\eta \colon\, |\eta| \le \delta \log L, \\ \eta(z) = 0 \,\forall\, z \in [-\ell_L, \ell_L]  }} \, 
\P_\eta \left( \bigcup_{i \in [0, M_L -1]} G^{(0,T_i)}_{T_1} \cap \Lambda^{(0,T_i)}_{T_1} \right) 
\ge 1 - c e^{-c^{-1}L^{\varepsilon}}.
\end{equation}
\end{lemma}
\begin{proof}
For $p_*$ is as in Lemma~\ref{l:permeability},
fix $\delta > 0$ such that $\delta \log \frac{1}{p_{*}} < \alpha - \beta$.
Fix $\eta$ with $|\eta| \le \delta \log L$, $\eta(z) = 0$ for all $z \in [-\ell_L, \ell_L]$. 

Put $\eta_t(x) := N(x,t)$ and use the Markov property to write, for $k \ge 0$,
\begin{align}\label{prlemcrossfinitetraps1}
& \P_\eta \left( \bigcap_{i=0}^{k+1} \left(G^{(0, T_i)}_{T_1} \cap \Lambda^{(0,T_i)}_{T_1}  \right)^c \cap \{\eta_{T_{i+1}}(0) = 0\} \right) \nonumber\\
 \le \; & \E_\eta \left[ \prod_{i=0}^{k} \mathbbm{1}_{ \left(G^{(0,T_i)}_{T_1} \cap \Lambda^{(0,T_i)}_{T_1}\right)^c \cap \{\eta_{T_{i+1}}(0) = 0\}} \; \P_{\eta_{T_{k+1}}} \left( (G_{T_1} \cap \Lambda_{T_1})^c \right) \right].
\end{align}
Since $|\eta_{T_{k+1}}| = |\eta| \le \delta \log L$ and 
$\eta_{T_{k+1}}(0) = 0$ inside the integral,
by Lemma~\ref{l:permeability} we may bound \eqref{prlemcrossfinitetraps1} from above by
\begin{equation}\label{prlemmacrossfinitetraps2}
\left(1-L^{\delta \log p_{*}} \right) \, \P_\eta \left(\bigcap_{i=0}^{k} \left(G^{(0, T_i)}_{T_1} \cap \Lambda^{(0,T_i)}_{T_1} \right)^c \cap \{\eta_{T_i}(0)=0\} \right).
\end{equation}
We conclude by induction that
\begin{align}\label{prlemmacrossfinitetraps3}
\P_\eta \left( \bigcap_{i=0}^{\lfloor M_L \rfloor - 1 } \left(G^{(0,T_i)}_{T_1} \cap \Lambda^{(0,T_i)}_{T_1} \right)^c \cap \{\eta_{T_i}(0) = 0\} \right)
& \le \left(1-L^{\delta \log p_{*}}\right)^{\lfloor M_L \rfloor}\nonumber\\
& \le c e^{- \frac{1}{c} L^{\varepsilon_*}}
\end{align}
where $\varepsilon_* := \alpha - \beta + \delta \log p_{*} > 0$ by our choice of $\delta$.
Now, using standard random walk estimates as in the proof of Lemma~\ref{l:noparticles}, we obtain
\begin{equation}\label{prlemmacrossfinitetraps4}
\P_\eta \left( \exists\, t \in [0, L^\alpha] \colon \eta_t(0) > 0 \right) \le c e^{-c^{-1}L^{\varepsilon'}}
\end{equation}
for some $\varepsilon' > 0$, so we may take $\varepsilon := \varepsilon' \wedge \varepsilon_*$.
\end{proof}

Finally, we gather all results of this section to prove Proposition~\ref{p:crossingtraps}.
\begin{proof}[Proof of Proposition~\ref{p:crossingtraps}]
Note that,
if $X_- \ge -L^\beta+1$, then $\Lambda^{(X_-,T_i)}_{T_1} \subset \{\bar{X}^{(X_-, T_i)}_{T_1} \ge L^\beta \}$. 
Therefore, by Lemma~\ref{l:estimateE-}, it is enough to show that
\begin{equation}\label{e:crossingtraps_reduced}
\P^{L^{-\frac{1}{16}}} \left( \bigcap_{i \in [0,M_L-1]} \left( G^{(X_-, T_i)}_{T_1} \cap \Lambda^{(X_-,T_i)}_{T_1} \right)^c \cap \{X_- \ge -L^\beta + 1 \}  \right) \le ce^{-c^{-1}(\log L)^2}.
\end{equation}
By a union bound and translation invariance, the left-hand side of \eqref{e:crossingtraps_reduced} is at most
\begin{align}\label{prpropcrosstraps0}
L^\beta \P^{L^{-\frac{1}{16}}} \left( \bigcap_{i \in [0,M_L-1]} \left( G^{(0,T_i)}_{T_1} \cap \Lambda^{(0,T_i)}_{T_1} \right)^c \cap \cE_L  \right)
\end{align}
where $\cE_L := \{N(z,0) = 0 \;\forall\, z \in [-\ell_L, \ell_L] \}$.

Recalling the definition of $T_i$, $\ell_L$ in \eqref{e:defT}, we note that, since all our random walks are $1$-Lipschitz, 
there exists $c_1 > 0$ such that the indicator functions of $G^{(0, T_i)}_{T_1}, \Lambda^{(0, T_i)}_{T_1}$ and $\cE_L$ 
are functionals of $U_A, N(A)$ with $A := [-c_1 L^\beta, c_1 L^\beta] \times [0,L^\alpha] \cap \Z^2$.

Let $B := \Z \setminus [-(c_1 +1)L^\beta, (c_1 + 1) L^\beta]$, put
\begin{equation}
\widehat{N}_L := \sum_{z \in B} \sum_{i \le N(z,0)} 
\mathbf{1}_{\{ \exists s \in [0,L^\alpha] \colon\, S^{z,i}_s \in [-c_1 L^\beta, c_1 L^{\beta}]\} }
\end{equation}
and, analogously to \eqref{defetaB},
\begin{equation}\label{defNB_prpropcrosstraps}
\eta^B(x) := \left\{\begin{array}{ll}
N(x,0) & \text{ if } x \notin B,\\
0 & \text{ otherwise.}
\end{array} \right.
\end{equation}
Lemmas~\ref{l:localeventslocalparticles} and \ref{l:crossfinitetraps} imply that
\begin{align}\label{prpropcrosstraps1}
& \P^{L^{-\frac{1}{16}}} \left( \bigcap_{i=0}^{\lfloor M_L \rfloor - 1 } \left(G^{(0, T_i)}_{T_1} \cap \Lambda^{(0,T_i)}_{T_1} \right)^c \cap \cE_L \right) \nonumber\\
& \qquad \quad \le \P_{L^{-\frac{1}{16}}} \left( \widehat{N}_L > 0 \right)
+ \E^{L^{-\frac{1}{16}}} \left[ \mathbbm{1}_{\cE_L} \P_{\eta^B} \left( \bigcap_{i=0}^{\lfloor M_L \rfloor - 1 } \left(G^{(0,T_i)}_{T_1} \cap \Lambda^{(0,T_i)}_{T_1} \right)^c \right) \right] \nonumber\\
& \qquad \quad \le \P^{L^{-\frac{1}{16}}} \left( \widehat{N}_L > 0 \right) + \P^{L^{-\frac{1}{16}}} \left( |\eta^B| > \delta \log L \right) + c e^{-c^{-1}L^\varepsilon}.
\end{align}
Reasoning as in the proof of Lemma~\ref{l:noparticles} (see \eqref{defNLx}--\eqref{e:estimparamNLx}),
we obtain
\begin{equation}\label{prpropcrosstraps2}
\P^{L^{-\frac{1}{16}}} \left( \widehat{N}_L > 0 \right) \le c e^{-c^{-1}L^{2\beta - \alpha}},
\end{equation}
while, since $|\eta^B|$ has under $\P^{L^{-\frac{1}{16}}}$ a Poisson law with
parameter at most $cL^{-(1/16-\beta)}$,
\begin{equation}\label{prpropcrosstraps3}
\P^{L^{-\frac{1}{16}}} \left( |\eta^B| > \delta \log L \right) \le \left( c L^{-(1/16 - \beta)} \right)^{\delta \log L} \le c e^{-c^{-1}(\log L)^2}.
\end{equation}
Combining \eqref{prpropcrosstraps0}--\eqref{prpropcrosstraps3}, we obtain \eqref{e:crossingtraps_reduced} and finish the proof.
\end{proof}

                                                                                                                                                                                                                                                                                                                                                                                                                                                                                                                                                                                                                                                                                                                                                                                                                                                                                                                                                                                                                                                                                                                                                                      \subsection{Perturbations of impermeable systems}
\label{ss:trigger_imperm}
In this section, we assume $q_0 = 0$.
As already mentioned, the main strategy in the proof of Theorem~\ref{t:triggerimperm}
is a comparison with an infection model, which we now describe.

Recall the random walks $S^{z,i}$ from Section~\ref{s:construction}.
Define recursively a random process $\xi(z,i,n) \in \{0,1\}$, $z \in \Z, i \in \N, n \in \N$
by setting
\begin{align}\label{e:initconf}
\begin{aligned}
  & \xi(z,i,0) = 1 \quad \text{ if } z \ge 0, z \in 2 \Z \text{ and } i \leq N(z,0),\\
  & \xi(z,i,0) = 0 \quad \text{ otherwise,}
\end{aligned}
\end{align}
and, supposing that $\xi(z,i,n)$ is defined for all $z \in \mathbb{Z}$, $i \in \N$,
\begin{equation}\label{e:defeta}
  \xi(z,i,n+1) = \left\{
\begin{array}{ll}
  1 & \begin{array}{l}
  \text{if } i \le N(z,0) \text{ and} \\  \exists \, z' \in \Z, i' \in \N \text{ with } \eta(z',i',n) = 1, S^{z',i'}_n = S^{z,i}_n,
  \end{array}\\
0 & \text{ otherwise.}
\end{array}\right.
\end{equation}
The interpretation is that, if $\xi(z,i,n) = 1$, then 
the particle $S^{z,i}$ is \emph{infected} at time $n$,
and otherwise it is \emph{healthy}.
Then \eqref{e:defeta} means that, whenever a group of particles shares a site at time $n$, 
if one of them is infected then all will be infected at time $n + 1$.

We are interested in the process $\bar{X} = (\bar{X}_n)_{n \in \Z_+}$ defined by 
\begin{equation}\label{e:defbarX}
  \bar{X}_n = \min \{S^{z,i}_n \colon\, z\in \Z, i \leq N(z,0) \text{ and } \xi(z,i,n) = 1\},
\end{equation}
i.e., $\bar{X}_n$ is the leftmost infected particle at time $n$.
We call $\bar{X}$ the \emph{front of the infection}.

Note that, by \eqref{e:initconf} and since $q_0=0$, all infected particles live on $2\Z$.
In particular, $\bar{X}_n \in 2\Z$ for all $n \ge 0$.
This implies the following.
\begin{lemma}\label{l:monot_infec}
If $p_\bullet = 0$, then $X_n \le \bar{X}_n$ for all $n \ge 0$.
\end{lemma}
\begin{proof}
Since the processes are one-dimensional, proceed by nearest-neighbour jumps,
are ordered at time $0$ and the difference in their positions lies in $2\Z$, 
we only need to consider what happens at times $s$ when $X_s = \bar{X}_s$.
For such times, $X_{s+1} = X_s - 1$ since $p_\bullet = 0$,
and thus $X_{s+1} \le \bar{X}_{s+1}$.
\end{proof}
The advantage of the comparison above becomes clear in light of the following.
\begin{proposition}  \label{p:infection}
  For any $\hat{\rho} > 0$, there exist $\hat{v}<0$, $c > 0$ such that
  \begin{equation}
    \P^{\hat{\rho}} \left(\bar{X}_L > \hat{v} L \right) \leq c\exp \left\{ - (\log L)^{3/2}/c \right\} \;\;\; \forall\; L \in \N.
  \end{equation}
\end{proposition}
\begin{proof}
Follows from Proposition~1.2 of \cite{manyRW}
once we map $2\Z$ to $\Z$ and apply reflection symmetry.
\end{proof}
We are now ready to finish the:
\begin{proof}[Proof of Theorem~\ref{t:triggerimperm}]
Fix $\hat{\rho} > 0$ and $\hat{L} \in \N$.
Suppose first that $p_\bullet =0$.
By Lemma~\ref{l:monot_infec} and Proposition~\ref{p:infection},
there exist $\hat{v}<0$, $c>0$ independent of $\hat{L}$ such that
\begin{align}\label{e:prbalimperm1}
\P^{\hat{\rho}} \left( X_{\hat{L}} > \hat{v} \hat{L} \right) \le \P^{\hat{\rho}} \left( \bar{X}_{\hat{L}} > \hat{v} \hat{L} \right) \le c e^{-(\log \hat{L})^{3/2}/c}.
\end{align}
Note now that, since $X_{\hat{L}}$ is supported in a finite space-time box,
the probability in the left-hand side of \eqref{e:prbalimperm1} is a continuous function
of $p_\bullet$.
Thus we can find $p_\star >0$ such that, if $p_\bullet \le p_\star$,
then \eqref{e:prbalimperm1} holds with $c$ replaced by $2c$,
concluding the proof.
\end{proof}



\section{Regeneration: proof of Theorem~\ref{t:limits_permeable}}
\label{s:reg_lowdensity}
In this section, we extend the results of Section 4 of \cite{HHSST14} to the case $v_\bullet < v_\circ$
and give the proof of Theorem~\ref{t:limits_permeable} under the conditions of item $a)$.

Fix $\rho > 0$. We assume that \eqref{e:BAL} holds with $v_\star > 0$ and some $\gamma > 1$.
We assume additionally that $p_\bullet > 0$.
In the sequel, we abbreviate $\P = \P^\rho$.

\begin{figure}[htbp]
\centering
\begin{tikzpicture}[scale=.4, font=\small]
\draw[fill, color=gray!40!white] (.5,1) -- (.5,5) -- (10,5) -- (8,.5) -- (.5,.5);
\draw[fill, color=gray!40!white] (10,5) -- (12,9.5) -- (20.5,9.5) -- (20.5,5) -- (10,5);
\foreach \x in {1,...,20} {
\foreach \y in {1,...,9} {
\draw[fill, color=gray!20!white] (\x,\y) circle [radius=0.05];}}
\foreach \y in {1,...,5} {
\pgfmathparse{9 - floor(2*(5-\y)/5)}
\xdef\z{\pgfmathresult}
\foreach \x in {1,...,\z} {
\draw[fill] (\x,\y) circle [radius=0.08];}}
\foreach \y in {5,...,9} {
\pgfmathparse{10 - floor(2*(5-\y)/5)}
\xdef\z{\pgfmathresult}
\foreach \x in {\z,...,20} {
\draw (\x,\y) circle [radius=0.08];}}
\node[below right] at (10,5) {$y$};
\end{tikzpicture}
\caption{An illustration of the sets $\um(y)$ (represented by white circles) and
\protect\rotatebox[origin=c]{180}{$\angle$}$(y)$ (represented by filled black circles),
with $y=(x,n) \in \Z^2$.}
\label{f:cones}
\end{figure}
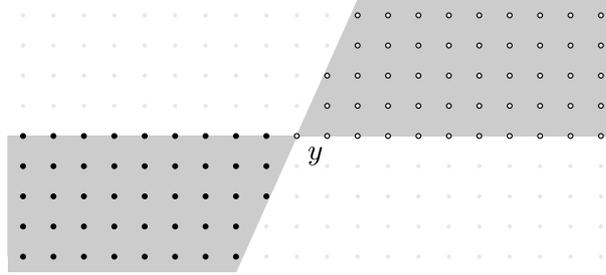

Define $\bar v = \tfrac13v_\star$.
For $x \in \R$ and $n \in \Z$, let $\um(x,n)$ be the cone in the first quadrant based
at $(x,n)$ with angle $\bar v$, i.e.,
\begin{equation}
\um(x,n) = \um(0,0) + (x,n), \text{ where } \um(0,0) = \{(x,n) \in \mathbb{Z}_+^2; x \geq \bar v n\},
\end{equation}
and $\tres(x,n)$ the cone in the third quadrant based at $(x,n)$ with angle $\bar v$, i.e.,
\begin{equation}
\tres(x,n) = \tres(0,0) + (x,n), \text{ where } \tres(0,0) = \{(x,n) \in \mathbb{Z}_-^2\colon\,x < \bar v n\}.
\end{equation}
(See Figure~\ref{f:cones}.) Note that $(0,0)$ belongs to $\um(0,0)$ but not to $\tres(0,0)$.

Fixed $y \in \Z^2$, define the following sets of trajectories in $W$:
\begin{equation}
\begin{aligned}
W_y^\um &= \text{ trajectories that intersect $\um(y)$ but not $\tres(y)$},\\
W_y^\tres &= \text{ trajectories that intersect $\tres(y)$ but not $\um(y)$},\\
W_y^\treze &= \text{ trajectories that intersect both $\um(y)$ and $\tres(y)$}.
\end{aligned}
\end{equation}
Note that $W^\um_y$, $W^\tres_y$ and $W^\treze_y$ form a partition of $W$. We write $Y_n$
to denote $Y^{0}_n$. For $y\in \Z^2$, define the sigma-algebras
\begin{equation}\label{e:sigmaalgebrastraj}
\mathcal{G}^{I}_{y} = \sigma \left( \omega(A) \colon\, A \subset W^{I}_{y}, A \in \cW  \right),
I = \um,\tres,\treze,
\end{equation}
and note that these are jointly independent under $\P$. Define also the sigma-algebras
\begin{equation}\label{e:sigmaalgebraunif}
\begin{aligned}
\mathcal{U}^{\um}_{y} & = \sigma \left( U_z \colon\, z \in \um(y) \right),\\
\mathcal{U}^{\tres}_{y} & = \sigma \left( U_{z} \colon\, z \in \tres(y) \right),
\end{aligned}
\end{equation}
and set
\begin{equation}
\label{e:sigmaalgebraFxt}
\mathcal{F}_{y} = \mathcal{G}^{\tres}_{y} \vee \mathcal{G}^{\treze}_{y} \vee \mathcal{U}^{\tres}_{y}.
\end{equation}

Next, define the \emph{record times}
\begin{equation}
\label{e:records}
R_k = \inf \{n\in{\mathbb{Z}_+}\colon\, X_n \ge (1-\bar{v})k + \bar{v} n \}, \qquad k \in \N,
\end{equation}
i.e., the time when the walk first enters the cone
\begin{equation}\label{e:defcones}
\um_k := \um((1-\bar{v})k,0).
\end{equation}
Note that, for any $k \in \N$, $y \in \um_k$ if and only if $y+(1,1) \in \um_{k+1}$. Thus $R_{k+1} \ge R_k+1$,
and $X_{R_k+1}-X_{R_k}=1$ if and only if $R_{k+1} = R_k+1$.

Define a filtration $\mathcal{F} = (\mathcal{F}_k)_{k \in \N}$ by setting
\begin{equation}
\label{e:filtration}
\mathcal{F}_k = \Big\{B \in \sigma(\omega,U) \colon\, \,
\forall \, y \in \Z^2, \, \exists \, B_{y} \in \mathcal{F}_{y}
\text{ s.t.}\, B \cap \{Y_{R_k} = y\} = B_{y} \cap \{Y_{R_k} = y\} \Big\},
\end{equation}
i.e., $\mathcal{F}_k$ is the sigma-algebra generated by $Y_{R_k}$, all $U_z$ with $z \in \tres(Y_{R_k})$
and all $\omega(A)$ such that $A \subset W^\tres_{Y_{R_k}} \cup W^\treze_{Y_{R_k}}$.
In particular, $(Y_i)_{0 \le i \le R_k} \in \mathcal{F}_k$.

Finally, define the event
\begin{equation}
\label{e:Axt}
A^{y} = \big\{Y^{y}_i \in \um(y) \,\,\forall\,i \in {\mathbb{Z}_+} \big\},
\end{equation}
in which the walker remains inside the cone $\um(y)$,
the probability measure
\begin{equation}
\mathbb{P}^{\um} (\cdot) = \mathbb{P} \left(~\cdot~ {\big |}
~\omega\big(W^{\treze}_{0}\big)=0,\, A^{0}\right),
\label{e:pmarrom}
\end{equation}
the \emph{regeneration record index}
\begin{equation}
\label{e:regrec}
\mathcal{I} = \inf\Big\{ k \in \N \colon\, \omega\big(W^{\treze}_{Y_{R_k}}\big)= 0,\,
A^{Y_{R_k}} \text{ occurs } \Big\}
\end{equation}
and the \emph{regeneration time}
\begin{equation}
\label{e:regtime}
\tau = R_{\mathcal{I}}.
\end{equation}

The following two theorems are our key results for the regeneration time.
\begin{theorem}
\label{t:regeneration}
Almost surely on the event $\{\tau < \infty\}$, the process $(Y_{\tau+i} - Y_\tau)_{i \in\Z_+}$
under either the law $\mathbb{P} (~\cdot \mid \tau, (Y_i)_{0\le i \le \tau})$ or $\mathbb{P}^{\um}
(~\cdot \mid \tau, (Y_i)_{0 \le i \le \tau})$ has the same distribution as that of $(Y_i)_{i \in\Z_+}$
under $\mathbb{P}^\um(\cdot)$.
\end{theorem}

\newconstant{c:tailreg}

\begin{theorem}
\label{t:tailregeneration}
There exists a constant $\useconstant{c:tailreg} > 0$ such that
\begin{equation}
\E \left[e^{\useconstant{c:tailreg} (\log \tau)^\gamma} \right] < \infty
\end{equation}
and the same holds under $\mathbb{P}^\um$.
\end{theorem}

Theorem~\ref{t:regeneration} is proved exactly as in \cite{HHSST14}.
Theorem~\ref{t:tailregeneration} was proved in \cite{HHSST14} in the non-nestling case and in the case $v_\bullet\geq v_\circ$. 
In the following section, we will fill the remaining gap by showing that it also holds when $v_\circ > 0 \ge v_\bullet$.

We may now conclude the:
\begin{proof}[Proof of Theorem~\ref{t:limits_permeable}]
One may follow word for word the proof of Theorem~1.4 in \cite{HHSST14} (Section~4.3 therein).
\end{proof}


\subsection{Proof of Theorem~\ref{t:tailregeneration}}
\label{ss:prooftail}
In what follows, constants may depend on $v_\circ$, $v_\bullet$, $v_\star$ and $\rho$.

Define the \emph{influence field at a point $y \in \Z^2$} as
\begin{equation}
\label{e:hxt}
h(y) = \inf \Big\{ l \in {\mathbb{Z}_+}\colon\, \omega(W^{\treze}_y \cap W^{\treze}_{y + (l,l)}) = 0 \Big\}.
\end{equation}

\newconstant{c:h_xt1}
\newconstant{c:h_xt2}

\begin{lemma}[Lemma 4.3 of \cite{HHSST14}]
\label{l:hxt_exp}
There exist constants $\useconstant{c:h_xt1}, \useconstant{c:h_xt2} > 0$ (depending on
$v_\star, \rho$ only) such that, for all $y \in \Z^2$,
\begin{equation}
\label{e:h_xt_exp}
\mathbb{P}[h(y) > l] \leq \useconstant{c:h_xt1} e^{-\useconstant{c:h_xt2} l}, \qquad l\in{\mathbb{Z}_+}.
\end{equation}
\end{lemma}

Set
\begin{equation}
\label{e:eps}
\delta = \frac{1}{4 \log\big( \tfrac{1}{p_{\bullet}} \big)}, \qquad
\epsilon = \frac{1}{4}(\useconstant{c:h_xt2} \delta \wedge 1),
\end{equation}
and put, for $T>1$,
\begin{equation}
T' = \lfloor T^\epsilon \rfloor,\qquad
T'' = \lfloor \delta \log T \rfloor.
\end{equation}

Define the \emph{local influence field at $(x,n)$} as
\begin{equation}
\label{e:hxt_local}
h^T(x,n) = \inf \big\{ l \in{\mathbb{Z}_+}\colon\, \omega( W^\um_{x-\lfloor(1-\bar{v}) \rfloor T',n} \cap W^{\treze}_{x,n}
\cap W^\treze_{x + l, n + l}) = 0 \big\}.
\end{equation}
Then we have the following.
\begin{lemma}[Lemma 4.4 of \cite{HHSST14}]
\label{l:locinfl}
For all $T > 1$ it holds $\P$-a.s.\ that
\begin{equation}
\label{e:locinfl}
\mathbb{P} \left( h^T(y) > l \;\middle|\; \mathcal{F}_{y-(\lfloor (1-\bar{v}) \rfloor T',0)} \right)
\le \useconstant{c:h_xt1} e^{-\useconstant{c:h_xt2} l}
\quad \forall \,y \in \Z^2,\,l \in \Z_+,
\end{equation}
where $\useconstant{c:h_xt1}, \useconstant{c:h_xt2}$ are the same constants of Lemma~\ref{l:hxt_exp}.
\end{lemma}

For $y \in \Z^2$, denote by
\begin{equation}\label{e:defkappay}
\kappa(y) := \max \{ k \in \N \colon\, y \in \um_k\}
\end{equation}
the index of the last cone containing $y$. Note that $\kappa(Y_{R_k}) = k$.
Then define, for $t \in \N$, the space-time parallelogram
\begin{equation}
\label{defparallel}
\mathcal{P}_t(y) = \left( \um(y) \setminus \um_{\kappa(y)+t} \right) \cap \left( y + \{(x,n) \in \Z^2 \colon\, n \le t/\bar{v} \} \right)
\end{equation}
and its right boundary
\begin{equation}
\label{defrightbound}
\partial^+\mathcal{P}_t(y) = \{z \in \Z^2  \setminus \mathcal{P}_t(y) \colon\, z - (1,0)
\in \mathcal{P}_t(y) \}.
\end{equation}
We say that ``$Y^y \text{ exits } \mathcal{P}_t(y) \text{ through the right}$'' when the
first time $i$ at which $Y^y_i \notin \mathcal{P}_t(y)$ satisfies $Y^y_i \in \partial^+
\mathcal{P}_t(y)$. Note that, if $y = Y_{R_k}$, this implies $Y^y_i = Y_{R_{k+t}}$.

In order to adapt the argument in \cite{HHSST14}, we will need to modify the definition of good record times
given there. For this, we need some additional definitions.

For $y \in \Z^2$, let
\begin{equation}\label{e:tildeW}
\widetilde{W}_y := \bigcup_{z \in \partial^+ \mathcal{P}_{T'}(y)} W^{\um}_{z-(\lfloor (1-\bar{v}) \rfloor T',0)} \cap W^{\treze}_z \cap W^{\treze}_{z + (T'', T'')}
\end{equation}
and, for $y_1,y_2 \in \Z^2$, denote by $\widetilde{\mathcal{T}}_{y_1,y_2}$ the trace of all trajectories in $\omega$ that do not belong to $\widetilde{W}_{y_1}$
or intersect $\tres(y_2)$.
Let $\widetilde{Y}^{y_1,y_2}$ be the analogous of $Y^{y_2}$ defined using $\widetilde{\mathcal{T}}_{y_1,y_2}$ instead of $\mathcal{T}$.
Note that, since $v_\circ > v_\bullet$, by monotonicity we have $\widetilde{X}^{y_1,y_2}_t \ge X^{y_2}_t$ for all $y_1,y_2 \in \Z^2$ and $t \in \Z_+$.

We say that $R_k$ is a \emph{good record time (g.r.t.)} when
\begin{align}
\label{e:good_record1}
& h^T(y) \leq T'' \;\; \forall\; y \in \partial^+ \mathcal{P}_{T'}(Y_{R_{k-T'}}),\\
\label{e:good_record2}
& U_{Y_{R_k} + (l,l)} \leq p_{\bullet} \quad \forall\,l = 0, \dots, T''-1,\\
\label{e:good_record3}
& \omega(W^\um_{Y_{R_{k}}} \cap W^{\treze}_{Y_{R_{k}+(T'',T'')}}
) = 0,\\
\label{e:good_record4}
& \widetilde{Y}^{k} \text{ exits } \mathcal{P}_{T'}(Y_{R_{k}+(T'',T'')}) \text{ through the right},
\end{align}
where $\widetilde{Y}^k := \widetilde{Y}^{y_1,y_2}$
with $y_1 = Y_{R_{k-T'}}$, $y_2 = Y_{R_{k}+(T'', T'')}$.
Note that \eqref{e:good_record1} is the same as $\{\omega(\widetilde{W}_{Y_{R_{k-T'}}})=0\}$ and that,
 when \eqref{e:good_record2} happens, $Y_{R_{k+T''}} = Y_{R_k}+(T'', T'')$.

The main differences with respect to the analogous definition in \cite{HHSST14} are:
\begin{enumerate}
\item In \eqref{e:good_record1}, we require a small local field not exactly at $Y_{R_k}$
but in every point of $\partial^+ \mathcal{P}_{T'}(Y_{R_{k-T'}})$, a set to which $Y_{R_k}$ belongs with large probability.
\item We do not require \eqref{e:good_record4} for $Y$ but only for $\widetilde{Y}$;
 we will see that, if the record time is good, then the same holds for $Y$ with large probability.
\end{enumerate}

We will need the following consequence of \eqref{e:BAL}.
\begin{lemma}\label{l:neverreturn}
\begin{equation}
\P \left(X_n \ge n v_\star \;\forall\; n \in \Z_+  \right) > 0.
\end{equation}
\end{lemma}
\begin{proof}
Fix $L > 1$ large enough such that
\begin{equation}
\label{e:tright3}
\mathbb{P} \left(\exists\,n \in {\mathbb{Z}_+}\colon\,X_n <
n v_\star - L(1-v_\star) \right)
\leq \frac12,
\end{equation}
which is possible by \eqref{e:BAL}. If $t > L$, then
\begin{align}
\label{e:tright4}
& \P \left(X_n \ge n v_\star \;\forall\; n \in \Z_+  \right) \nonumber\\
& \ge \P \left( U_{(i,i)} \le p_\bullet \; \forall \; i = 0, \ldots, L-1, X^{(L,L)}_n - L \ge n v_\star - (1-v_\star)L \;\forall\; n \in \Z_+ \right) \nonumber\\
& = p_\bullet^L
\left\{ 1 - \P \left(\exists\,n \in {\mathbb{Z}_+}\colon\, X_n < nv_\star - (1-v_\star)L \right) \right\} \nonumber\\
& \ge \tfrac12 p_\bullet^L > 0
\end{align}
as desired.
\end{proof}

As in \cite{HHSST14}, the following proposition is the main step to control the tail of the regeneration time.

\newconstant{c:manygrts}

\begin{proposition}
\label{p:manygrts}
There exists a constant $\useconstant{c:manygrts} >0$ such that, for all $T>1$ large enough,
\begin{equation}
\mathbb{P}\left[\text{$R_k$ is not a g.r.t.\ for all $1\le k \le T$ }\right]
\leq e^{ -\useconstant{c:manygrts} \sqrt{T}}.
\end{equation}
\end{proposition}

\begin{proof}
First we claim that there exists a $c>0$ such that, for any $k \ge T'$,
\begin{equation}
\label{e:saw_pemba}
\mathbb{P} \left[ R_k \text{ is a g.r.t.} \big| \mathcal{F}_{k-T'} \right]
\geq c T^{\delta \log(p_{\bullet})}  \text{ a.s.}
\end{equation}
To prove \eqref{e:saw_pemba}, we will find $c>0$ such that
\begin{align}
\label{e:good_cond1}
& \mathbb{P} \big[ \text{\eqref{e:good_record1}} \; \big| \; \mathcal{F}_{k-T'} \big] \geq c
 & \text{ a.s.,}\\
\label{e:good_cond2}
& \mathbb{P} \big[ \text{\eqref{e:good_record2}} \; \big| \; \text{\eqref{e:good_record1}}, \mathcal{F}_{k-T'} \big]
\geq T^{\delta \log(p_{\bullet})}  & \text{ a.s.,}\\
\label{e:good_cond3}
& \mathbb{P} \big[ \text{\eqref{e:good_record3}} \; \big| \; \text{\eqref{e:good_record1}}, \text{\eqref{e:good_record2}},
\mathcal{F}_{k-T'} \big] \geq c  & \text{ a.s.,}\\
\label{e:good_cond4}
& \mathbb{P} \big[ \text{\eqref{e:good_record4}} \; \big| \; \text{\eqref{e:good_record1}}, \text{\eqref{e:good_record2}}, \text{\eqref{e:good_record3}}, \mathcal{F}_{k-T'} \big] \geq c  & \text{ a.s.\ }
\end{align}

\eqref{e:good_cond1}: Fix $B \in \mathcal{F}_{k-T'}$. 
Summing over the values of $Y_{R_{k-T'}}$ and using a union bound we may write
\begin{align}
\label{e:manygrts1}
\mathbb{P} \left( \text{\eqref{e:good_record1}}^c, B \right)
& \le \sum_{y_1 \in \Z^2} \sum_{y_2 \in \partial^+\mathcal{P}_{T'}(y_1)} \mathbb{P} \left(h^T(y_2) > T'', Y_{R_{k-T'}}
= y_1,  B_{y_1} \right).
\end{align}
Noting that $y_2 - (\lfloor (1-\bar{v}) \rfloor) T', 0) - y_1 \in \Z_+^2$ for large enough $T$,
we may use Lemma~\ref{l:locinfl} and $|\partial^+ \mathcal{P}_t(y)| \le t/\bar{v}$ 
to further bound \eqref{e:manygrts1} by
\begin{align}
\label{e:manygrts2}
\frac{\useconstant{c:h_xt1}}{\bar{v}} T' e^{-\useconstant{c:h_xt2} T''}   \mathbb{P} \left( B \right)
\le  \frac{\useconstant{c:h_xt1}}{\bar{v}} e^{\useconstant{c:h_xt2}} T^{-\frac34 \delta
\useconstant{c:h_xt2}}   \mathbb{P} \left( B \right)
\end{align}
where the last inequality uses the definition of $\epsilon$.
Thus, for $T$ large enough, \eqref{e:good_cond1} is satisfied with e.g.\ $c = 1/2$.

\eqref{e:good_cond2}: This follows from the fact that $(U_{Y_{R_k}+(l,l)})_{l \in \N_0}$ is
independent of the sigma-algebra $\sigma(\omega(A) \colon\, A \subset \widetilde{W}_{Y_{R_{k-T'}}}) \vee \mathcal{F}_{k}$
with respect to which \eqref{e:good_record1} is measurable.

\eqref{e:good_cond3}: We may ignore the conditioning on \eqref{e:good_record2} since
this event is independent of the others.
Since \eqref{e:good_record1} is equivalent to $\omega(\widetilde{W}_{Y_{R_{k-T'}}})=0$,
for $B \in \mathcal{F}_{k-T'}$ we may write
\begin{align}
\label{e:manygrts3}
& \mathbb{P} \left(\text{\eqref{e:good_record3}}, \text{\eqref{e:good_record1}}, B\right)
= \mathbb{P} \left(\omega(W^{\um}_{Y_{R_k}} \cap W^{\treze}_{Y_{R_{k+T''}}} \setminus \widetilde{W}_{Y_{R_{k-T'}}}) =0, \text{\eqref{e:good_record1}}, B\right) \nonumber\\
& = \sum_{\stackrel{y_1,y_2 \in \Z^2 \colon}{y_2 - y_1 \in \N^2}} \mathbb{P} \left(\omega(W^{\um}_{y_2}
\cap W^{\treze}_{y_2 + (T'',T'')} \setminus \widetilde{W}_{y_1} ) =0, Y_{R_k}=y_2, Y_{R_{k-T'}} = y_1, \omega(\widetilde{W}_{y_1})=0, B_{y_1} \right) \nonumber\\
&= \sum_{\stackrel{y_1,y_2 \in \Z^2 \colon}{y_2 - y_1 \in \N^2}} \mathbb{P} \left(\omega(W^{\um}_{y_2}
\cap W^{\treze}_{y_2 + (T'',T'')} \setminus \widetilde{W}_{y_1} ) =0 \right)
\nonumber\\[-20pt]
& \qquad \qquad \qquad \qquad \qquad \qquad \qquad \times \mathbb{P}\left( Y_{R_k} = y_2, Y_{R_{k-T'}} = y_1, \omega(\widetilde{W}_{y_1})=0, B_{y_1}\right) \nonumber\\
&  \ge \mathbb{P}\left(\omega(W^{\treze}_0) = 0 \right) \mathbb{P} \left(\text{\eqref{e:good_record1}}, B \right),
\end{align}
where the second equality uses the independence between $\sigma(\omega(A) \colon\, A \subset W^{\um}_{y_2} \setminus \widetilde{W}_{y_1})$ and $\mathcal{F}_{y_2} \vee \sigma(\omega(A) \colon\, A \subset \widetilde{W}_{y_1})$, and the last step uses the monotonicity and translation invariance of $\omega$.

\eqref{e:good_cond4}:
We may again ignore \eqref{e:good_record2} in the conditioning since this event is independent of all the others.
Note that $\text{\eqref{e:good_record1}} \cap \{Y_{R_{k-T'}} = y\} = \text{\eqref{e:good_record1}}_y \cap \{Y_{R_{k-T'}} = y\}$
where $\text{\eqref{e:good_record1}}_y \in \sigma(\omega(A) \colon\, A \subset \widetilde{W}_y)$,
and similarly $\eqref{e:good_record3} \cap \{Y_{R_k}=y\} = \eqref{e:good_record3}_y \cap \{Y_{R_k}=y\}$
with $\text{\eqref{e:good_record3}}_y \in \mathcal{F}_{y + (T'', T'')}$.
Now take $B \in \mathcal{F}_{k-T'}$ and write
\begin{align}
\label{e:manygrts4}
& \mathbb{P} \left( \text{\eqref{e:good_record4}}, \text{\eqref{e:good_record3}}, \text{\eqref{e:good_record1}}, B \right) \nonumber\\
& = \sum_{\stackrel{y_1,y_2 \in \Z^2 \colon}{y_2 - y_1 \in \N^2}}
\mathbb{P} \Big( \widetilde{Y}^{y_1,y_2+(T'',T'')} \text{ exits } \mathcal{P}_{T'}(y_2+(T'', T'')) \text{ through the right}, \nonumber\\[-25pt]
& \qquad \qquad \qquad \qquad \qquad \qquad \;\;\; Y_{R_k}=y_2, Y_{R_{k-T'}} = y_1, \text{\eqref{e:good_record3}}_{y_2}, \text{\eqref{e:good_record1}}_{y_1}, B_{y_1} \Big).
\end{align}
Since $\widetilde{Y}^{y,z}$ is independent of $\mathcal{F}_{z} \vee \sigma(\omega(A) \colon\, A \subset \widetilde{W}_{y})$,
the last line equals
\begin{align}
\label{e:manygrts5}
& \sum_{\stackrel{y_1,y_2 \in \Z^2 \colon}{y_2 - y_1 \in \N^2}}
\mathbb{P} \Big( \widetilde{Y}^{y_1,y_2+(T'',T'')} \text{ exits } \mathcal{P}_{T'}(y_2+(T'', T'')) \text{ through the right} \Big) \nonumber\\[-20pt]
& \qquad \qquad \qquad \qquad \;\, \times \P \Big( Y_{R_k}=y_2, Y_{R_{k-T'}} = y_1, \text{\eqref{e:good_record3}}_{y_2}, \text{\eqref{e:good_record1}}_{y_1}, B_{y_1} \Big)
\nonumber\\
& \ge \mathbb{P}\left( X_n \ge n v_\star \;\forall\; n \in \Z_+ \right)
\P \left( \text{\eqref{e:good_record3}}, \text{\eqref{e:good_record1}}, B \right),
\end{align}
where for the last step we use $\widetilde{X}^{y,z}_{t} \ge X^{z}_t$ and translation invariance.
Now \eqref{e:good_cond4} follows from \eqref{e:manygrts5} and Lemma~\ref{l:neverreturn}.

Thus, \eqref{e:saw_pemba} is verified.
To conclude, note that $\{R_k \text{ is a g.r.t.}\} \in \mathcal{F}_{k+ \bar{c} T'}$
for some $\bar{c} \in \N$ independent of $T$.
Indeed, this can be verified for each \eqref{e:good_record1}--\eqref{e:good_record4}
using the observation that, if an event $A \in \mathcal{F}_\infty$ satisfies $A \cap \{Y_{R_k}=y\} = A_y \cap \{Y_{R_k}=y\}$
with $A_y \in \mathcal{F}_{y+(t,t)}$, then $A \in \mathcal{F}_{k+t+1}$.
Hence we obtain
\begin{align}
\label{e:manygrts6}
& \mathbb{P} \left( R_k \text{ is not a g.r.t.\ for any } k \le T \right) \nonumber\\
& \qquad \le \mathbb{P} \left( R_{(\bar{c}+1)kT'} \text{ is not a g.r.t.\ for any } k \le \frac{T}{(\bar{c}+1)T'} \right) \nonumber\\
& \qquad \le \exp \left\{ -\frac{c}{\bar{c}+1} \frac{T^{1+{\delta \log(p_\circ \wedge p_\bullet)}}}{T'} \right\}
\le \exp \left\{-\frac{c}{\bar{c}+1}T^{\frac12} \right\}
\end{align}
by our choice of $\epsilon$ and $\delta$.
\end{proof}

To prove Theorem~{\rm \ref{t:tailregeneration}},
we can now proceed as in the proof of Theorem 4.2 in \cite{HHSST14},
with a few modifications as follows.
Defining the events $E_1$ and $E_2$ as in equation (4.52) therein,
we may assume that $R_{\lfloor \bar{v} T \rfloor +T''+T'}\le T$
and that $R_k$ is a g.r.t.\ with $k \le \bar{v} T$.
To show that $\omega(W^{\treze}_{Y_{R_{k+T''}}})=0$,
we may use the same arguments therein once we note that, on $E_2^c$,
$Y_{R_k} \in \partial^+ \mathcal{P}_{T'}(Y_{R_{k-T'}})$.
From this together with \eqref{e:good_record1} it follows that
$\widetilde{\mathcal{T}}_{y,z}$ coincides with $\mathcal{T}$ inside $\um(z)$,
where $y=Y_{R_{k-T'}}$ and $z=Y_{R_{k+T''}}$.
On $E_2^c$ this implies that $Y^{z}_t=\widetilde{Y}^k_t \in \um(z)$ for all $t \in \Z_+$,
i.e., $A^{Y_{R_{k+T''}}}$ occurs.
Thus $\tau \le R_{k+T''} \le T$, and the proof is concluded as before.


\bibliographystyle{plain}
\bibliography{all}

\end{document}